\documentclass[11pt,a4paper,twoside]{article}

\usepackage{iftex}
\usepackage[T1]{fontenc}
\ifPDFTeX
  \usepackage[utf8]{inputenc}
\fi
\usepackage{lmodern}
\usepackage{indentfirst}
\usepackage{amsmath,amsthm,amssymb,mathtools}

\usepackage{geometry}
\usepackage{xcolor}
\usepackage{microtype}
\usepackage{fancyhdr}

\usepackage[
  colorlinks=true,
  linkcolor=blue!45!black,
  citecolor=blue!45!black,
  urlcolor=blue!45!black
]{hyperref}

\allowdisplaybreaks[3]
\newtheoremstyle{paperplain}{0.65em}{0.65em}{\itshape}{}%
  {\bfseries}{.}{0.5em}{}
\theoremstyle{paperplain}
\newtheorem{theorem}{Theorem}[section]
\newtheorem{proposition}[theorem]{Proposition}
\newtheorem{lemma}[theorem]{Lemma}

\newtheoremstyle{paperremark}{0.65em}{0.65em}{\normalfont}{}%
  {\bfseries}{.}{0.5em}{}
\theoremstyle{paperremark}
\newtheorem{remark}[theorem]{Remark}

\newcommand{\runningtitle}{}
\let\originaltitle\title
\renewcommand{\title}[2][]{%
  \originaltitle{#2}%
  \if\relax\detokenize{#1}\relax
    \gdef\runningtitle{#2}%
  \else
    \gdef\runningtitle{#1}%
  \fi
}

\title[Graphical curvature quotient equations]{\textbf{Interior curvature estimates for graphical curvature quotient equations}}
\author{%
  \textbf{Fei Han}
  \and
  \textbf{Yan Ma}\thanks{Corresponding author.}
}

\date{}
\makeatletter
\newcommand{\printfirstpagebottommatter}{%
  \begingroup
  \renewcommand{\thefootnote}{}%
  \renewcommand{\@makefntext}[1]{\noindent##1}%
  \footnotetext[0]{%
    \textit{Keywords.}
    Curvature quotient equation; interior curvature estimate;
    doubling method; two-surface comparison; Pogorelov estimate.
    \par\smallskip
    \textit{2020 Mathematics Subject Classification.}
    35J60, 35B45, 53C42.%
  }%
  \endgroup
}
\makeatother
\newcommand{\printauthorinformation}{%
  \par\bigskip
  \begingroup
  \small
  \setlength{\parindent}{0pt}%

  \noindent
  \textsc{Fei Han}: \href{mailto:137823121@qq.com}{137823121@qq.com}\par
  School of Mathematics Sciences, Xinjiang Normal University,
  No.~102, Xinyi Road, Shayibake District, Urumqi, 830054,
  Xinjiang Uygur Autonomous Region, China

  \par\medskip

  \noindent
  \textsc{Yan Ma}: \href{mailto:Mayan94743369@163.com}{Mayan94743369@163.com}\par
  Department of Mathematics, College of Science, Shihezi University, 
  No.~221, Beisi Road, Shihezi, 832003,
  Xinjiang Uygur Autonomous Region, China\par
  \textit{Corresponding author.}

  \endgroup
}

\numberwithin{equation}{section}

\hypersetup{pdftitle={Interior curvature estimates for graphical curvature quotient equations},
pdfauthor={Fei Han and Yan Ma}}
\begin{document}
\maketitle
\printfirstpagebottommatter
\begin{abstract}
We prove interior curvature estimates for admissible graphical solutions of curvature quotient equations in the range $3\leq k<n$ on the full G{\aa}rding cone $\Gamma_k$. No convexity, semiconvexity, or $\Gamma_{k+1}$-admissibility assumption is required. The proof combines a new quantitative compression inequality, a two-surface comparison and doubling argument, and a Pogorelov estimate.
\end{abstract}

\section{Introduction}\label{sec:introduction}

Let $X(x)=(x,u(x))$ be a hypersurface graph in $\mathbb R^{n+1}$.
We consider the curvature quotient equation
\begin{equation}\label{1.1}
    \frac{\sigma_k(\kappa[u])}{\sigma_{k-2}(\kappa[u])}=1,
    \qquad
    \kappa[u]\in\Gamma_k,
    \qquad
    3\leq k<n,
\end{equation}
where $\kappa[u]$ denotes the principal curvatures with respect to the
downward unit normal, $\sigma_j$ is the unnormalized $j$th elementary
symmetric polynomial, $\sigma_0=1$, and
\[
    \Gamma_k
    :=
    \left\{
        \kappa\in\mathbb R^n:
        \sigma_j(\kappa)>0,\ 1\leq j\leq k
    \right\}
\]
is the G{\aa}rding cone. A solution is called admissible if $\kappa[u]\in\Gamma_k$. We set
\[
    F(\kappa)
    :=
    \left(
        \frac{\sigma_k(\kappa)}{\sigma_{k-2}(\kappa)}
    \right)^{\frac12}.
\]
The operator $F$ is elliptic and concave on $\Gamma_k$ and is
homogeneous of degree one; see \cite{CNSHessian,LTineq}. Under a fixed
gradient bound, a curvature bound therefore gives uniform ellipticity
for the graphical equation.

Curvature quotient equations arise naturally in the study of fully
nonlinear curvature problems and belong to the same broad family as
Hessian quotient equations. The boundary value theory for nonlinear
Hessian and curvature equations was developed by Caffarelli, Nirenberg,
and Spruck \cite{CNSHessian,CNS}. For curvature quotient equations,
Lin and Trudinger \cite{LT} treated homogeneous boundary data, while
Ivochkina, Lin, and Trudinger \cite{ILT} considered general boundary
values. Ivochkina \cite{IvoNormal} obtained local estimates for the
boundary normal second derivative. These results provide the basic
Dirichlet theory, but purely interior second-derivative estimates are
more delicate because they must be independent of second derivatives of
the boundary data. Sheng, Urbas, and Wang \cite{ShengUrbasWang}
established interior curvature bounds for admissible solutions of a class
of curvature equations subject to affine Dirichlet data, extending
Pogorelov-type estimates to hypersurfaces. Their boundary-dependent
setting is an important precursor to the purely interior problem here.

The purely interior theory for the quadratic Hessian equation
$\sigma_2(D^2u)=1$ provides a useful point of comparison. Warren and
Yuan \cite{WarrenYuan} established the three-dimensional estimate
through special Lagrangian geometry; Qiu \cite{QiuHessian} treated
more general positive right-hand sides in dimension three and developed
a maximum-principle doubling argument. Shankar and Yuan
\cite{ShankarYuan} resolved the four-dimensional constant equation,
and Li and Wu \cite{LiWuHessian} recently obtained the estimate in
every dimension under natural $\Gamma_2$-admissibility. For Hessian
quotients, interior estimates under convexity were established in
\cite{Lu,LuTsai,LiWuQuotient}, and related results under
semiconvexity in \cite{MeiYan,DongZhang}. Qiu and Yan \cite{QiuYanHessian} gave a complete solution to the interior estimate problem for Hessian quotient equations in all dimensions under natural admissibility, for constant Hessian quotients with degree gap one or two. Nevertheless, Hessian quotient equations
and graphical curvature quotient equations have different geometric
structures. For a Hessian equation, the nonlinear operator
acts directly on the eigenvalues of $D^2u$, whereas for a graphical
curvature equation the principal curvatures are the eigenvalues of the
shape operator, which depends simultaneously on $D^2u$, $Du$, and the
induced metric. Therefore, the interior Hessian estimates for quotient
equations do not directly imply interior curvature estimates for
\eqref{1.1}.

For graphical curvature quotients, Liu \cite[Theorem~1.1]{Liu} proved
an interior estimate for
$\frac{\sigma_n}{\sigma_{n-2}}$ on convex graphs. Wang
\cite[Theorem~1.2 and Remark~1.5]{Wang} treated
$\frac{\sigma_k}{\sigma_{k-2}}$ under strict convexity and showed that
the argument also applies under the stronger condition
$\kappa[u]\in\Gamma_{k+1}$. When $(n,k)=(4,3)$, equation \eqref{1.1}
is equivalent on $\Gamma_3$ to the critical-phase special Lagrangian
curvature equation
\[
    \sum_{i=1}^4\arctan\kappa_i=\pi,
\]
and the corresponding estimate follows from Qiu and Zhou
\cite{QiuZhou}.

The closest geometric precedent is the graphical scalar curvature
equation. Guan and Qiu \cite{GuanQiu} obtained interior curvature
estimates under an additional lower bound for $\sigma_3(\kappa)$, and
Qiu \cite{Qiu3} removed this condition in dimension three. Most recently, Qiu and Yan \cite{QiuYan} established the purely
interior estimate for $\sigma_2(\kappa[u])=1$ on the full
$\Gamma_2$ cone in every dimension, without convexity, semiconvexity,
or a higher-cone assumption. Their work introduced a geometric
comparison between distinct points on two admissible graphs: a
Jacobi inequality for their distance produces an interior separation
seed, a modified doubling maximum principle propagates the seed,
and a genuinely two-surface Pogorelov estimate uses second-order
information at both points. In particular, their two-point framework
overcomes the mismatch of the induced metrics and normals that
obstructs a direct comparison of the vertical gap at one base point.
This new mechanism is the foundation of our approach to intermediate
graphical curvature quotient equations. Two additional difficulties arise that
are absent in the scalar-curvature case. First, the matrix produced by
the two-point second variation is a rank-one compression of the
curvature matrix. For $\sigma_2$, the required positivity can be read
directly from the scalar-curvature structure, whereas for
$\frac{\sigma_k}{\sigma_{k-2}}$ one must prove that the compressed matrix
remains in $\Gamma_k$ and retains a fixed positive portion of the
quotient. Second, the propagation argument requires quantitative
ellipticity in a distinguished curvature direction. This is not
available directly from the higher quotient.  These two ingredients allow the two-surface method to be carried out
under the natural admissibility condition $\kappa[u]\in\Gamma_k$.

Our main result is the following quantitative interior estimate.

\begin{theorem}\label{thm:main}
Let $3\leq k<n$, $R>0$, and $K\geq1$. Suppose that
$u\in C^4(B_R)$ satisfies
\begin{equation}\label{1.2}
\begin{gathered}
    \kappa[u]\in\Gamma_k,
    \qquad
    \frac{\sigma_k(\kappa[u])}{\sigma_{k-2}(\kappa[u])}=1
    \quad\text{in }B_R,
    \\
    \|u\|_{L^\infty(B_R)}
    +\|Du\|_{L^\infty(B_R)}
    \leq K.
\end{gathered}
\end{equation}
Then
\begin{equation}\label{1.3}
    \sup_{B_{\frac{R}{2}}}|\kappa[u]|
    +
    \sup_{B_{\frac{R}{2}}}|D^2u|
    \leq C(n,k,R,K).
\end{equation}
\end{theorem}

No convexity, semiconvexity, or $\Gamma_{k+1}$ assumption is imposed.
The endpoint $k=n$ under convexity was treated by Liu \cite{Liu}.
We retain the radius $R$, since dilation changes the prescribed value
of the curvature quotient. The degree-gap range is sharp in the sense
that a purely interior estimate under natural admissibility cannot
hold for every quotient with a larger gap: Urbas' nonclassical
$\sigma_3$-curvature examples provide the first gap-three obstruction
\cite{UrbasNonclassical}. 

Our proof follows the geometric maximum principle mechanism of Qiu and Yan
\cite{QiuYan}. On a small ball we solve $F(\kappa[v])=s_0<1$ with
$v=u$ on the boundary and compare the two graphs through the one-sided
Euclidean distance $d$ from the comparison graph to the closed
hypograph of the original graph. The two-point second variation
produces a rank-one compression of the curvature matrix; a new
quantitative compression inequality keeps it in $\Gamma_k$ and
preserves a definite portion of the quotient. This yields a Jacobi
inequality for $-\log d$ with a positive $1/d$ term, supplying a
separation seed. After normalizing $\kappa[v]/s_0$, the quotient
identity gives the directional ellipticity needed to propagate this
seed by a modified doubling maximum principle. Finally, a product-space
Pogorelov maximum principle pairs the two tangent spaces isometrically
and controls the unfavorable angle and curvature terms. Uniform
separation then turns the weighted curvature bound into the asserted
interior estimate.

The paper is organized as follows. Section~\ref{geometry} fixes the
geometric conventions and records the basic identities used throughout
the paper. Section~\ref{sec:separation} introduces the comparison
graph, proves the interior gradient estimate, and establishes the
uniform separation estimate. Section~\ref{pog-section} proves the
two-surface Pogorelov estimate. Finally, Section~\ref{sec:main-proof}
constructs the comparison graph on a sufficiently small ball and
completes the proof of Theorem~\ref{thm:main}.

\section{Preliminaries}\label{geometry}
We first fix the geometric notation used throughout the paper.
Let $\Sigma$ be the graph of a smooth function $u$ over a domain
$\Omega\subset\mathbb{R}^n$. We write
\[
    X(x)=(x,u(x)),\qquad
    W=\sqrt{1+|Du|^2},\qquad
    \nu=\frac{(Du,-1)}{W},
\]
where $\nu$ is the downward unit normal and $D$ denotes
differentiation in the base variables. In graph coordinates,
\[
    g_{ij}=\delta_{ij}+D_i u\,D_j u,
    \qquad
    h_{ij}=\frac{D_{ij}u}{W}.
\]
We use the convention
$\overline{\nabla}_{X_i}\nu=h_i{}^jX_j$, where
$h_i{}^j=g^{j\ell}h_{i\ell}$.

Covariant derivatives with respect to the induced metric are denoted
by subscripts. Thus $h_{ijk}=\nabla_kh_{ij}$, and the Codazzi
equations give $h_{ijk}=h_{ikj}$. We identify $u$ with the height
function on $\Sigma$. At a fixed point we choose a principal
orthonormal frame so that
\[
    h_{ij}=\kappa_i\delta_{ij},
    \qquad
    \kappa_1\geq\cdots\geq\kappa_n.
\]

Let $F=F(\kappa)$ be a positive, symmetric, elliptic, concave
function, homogeneous of degree one. We write
\[
    F^{ij}=\frac{\partial F}{\partial h_{ij}},
    \qquad
    F_i=\frac{\partial F}{\partial\kappa_i},
\]
with the metric held fixed. In a principal frame,
$F^{ij}=F_i\delta_{ij}$. We use the notation
\[
    \Delta_F\phi=F^{ij}\phi_{ij},
    \qquad
    |\nabla\phi|_F^2=F^{ij}\phi_i\phi_j,
    \qquad
    \mathcal E=F^{ij}h_i{}^kh_{kj}.
\]

Let $E_{n+1}=(0,\ldots,0,-1)$. Then
$\langle\nu,E_{n+1}\rangle=W^{-1}$ and
$u_{ij}=W^{-1}h_{ij}$. If $F(\kappa)=s$ is constant, then
\begin{equation}\label{2.1}
\begin{aligned}
    \Delta_F\!\left(\frac1W\right)
    &=-\frac{\mathcal E}{W},
    &
    \left(\frac1W\right)_i
    &=-\kappa_i u_i,
    \\
    \Delta_Fu
    &=\frac{s}{W},
    &
    \sum_iF_i\kappa_i
    &=s.
\end{aligned}
\end{equation}
The second identity is written in a principal frame, with no
summation over $i$. Differentiating the equation and using Codazzi,
we also have
\[
    \sum_iF_i h_{iij}=0,
    \qquad
    1\leq j\leq n.
\]

For $1\leq r\leq n$, let
\[
    T_{r-1}^{ij}
    =
    \frac{\partial\sigma_r}{\partial h_{ij}}
\]
be the Newton tensor. In Euclidean ambient space,
the Newton tensors are divergence-free and satisfy
\begin{equation}\label{2.2}
\begin{aligned}
    \nabla_iT_{r-1}^{ij}
    &=0,
    &
    T_{r-1}^{ij}g_{ij}
    &=(n-r+1)\sigma_{r-1},
    \\
    T_{r-1}^{ij}h_{ij}
    &=r\sigma_r,
    &
    T_{r-1}^{ij}u_{ij}
    &=\frac{r\sigma_r}{W}.
\end{aligned}
\end{equation}
If $\kappa\in\Gamma_k$, then $T_{r-1}$ is positive definite for
$1\leq r\leq k$.

For two graph functions $u$ and $v$, the same notation is used with
subscripts $u$ and $v$ for the corresponding graph quantities. In
particular, $X_u,X_v$ denote the graph parametrizations and
$\Sigma_u,\Sigma_v$ their images. When the two graphs are introduced
simultaneously below, their downward unit normals will be denoted by
$N_u$ and $N_v$. Unless $D$ is explicitly used, derivatives on each
graph are taken with respect to its induced metric.

\section{Uniform separation}\label{sec:separation}
Throughout this section, set $R_0=\frac{3R}{4}$ and let $u$
satisfy \eqref{1.2}. Let
$v\in C^4(B_{R_0})\cap C^0(\overline{B}_{R_0})$ be an admissible
comparison solution of
\begin{equation}\label{3.1}
    \begin{cases}
        F(\kappa[v])=s_0
            & \text{in }B_{R_0},\\
        \kappa[v]\in\Gamma_k
            & \text{in }B_{R_0},\\
        v=u
            & \text{on }\partial B_{R_0}.
    \end{cases}
\end{equation}
Here $s_0=s_0(n,k)>0$ is the fixed constant chosen in
\eqref{3.8}. For simplicity, we write $s=s_0$ below.

The comparison principle gives $v>u$ in $B_{R_0}$.
Moreover, admissibility implies that the mean curvature of
$\Sigma_v$ is positive, so the maximum principle yields
\[
    v\leq\max_{\partial B_{R_0}}u
    \qquad\text{in }B_{R_0}.
\]
Consequently,
\[
    u<v\leq\max_{\partial B_{R_0}}u
    \qquad\text{in }B_{R_0},
\]
and
\[
    \operatorname{osc}_{B_{R_0}}v
    \leq
    \operatorname{osc}_{B_{R_0}}u
    \leq 2K.
\]

Before proving the separation estimate, we fix the notation for the two
graphs. Write
\[
    X_u(x)=(x,u(x)),\qquad X_v(x)=(x,v(x)),
    \qquad \Sigma_u=X_u(B_{R_0}),\qquad \Sigma_v=X_v(B_{R_0}),
\]
and set
\[
    W_u=\sqrt{1+|Du|^2},\qquad
    W_v=\sqrt{1+|Dv|^2}.
\]
The downward unit normals to $\Sigma_u$ and $\Sigma_v$ are denoted by
$N_u$ and $N_v$, respectively, and are given by
\begin{equation}\label{eq:section3-normals}
    N_u=\frac{(Du,-1)}{W_u},\qquad
    N_v=\frac{(Dv,-1)}{W_v}.
\end{equation}
Unless otherwise indicated, covariant derivatives on each graph are
taken with respect to its induced metric.

\subsection{An interior gradient estimate}\label{gradient-section}
We shall need an interior gradient estimate for the comparison graph.
Here and below, unless otherwise stated, $F$ is the quotient operator
in the introduction. We include a proof using an exponential cutoff,
as in the interior gradient method of \cite{Korevaar}. The estimate
depends on the oscillation, rather than the boundary slope.

\begin{lemma}\label{gradient}
Let $3\leq k\leq n$, and let $v\in C^3(B_r)$ be an admissible
solution of $F(\kappa[v])=s>0$. If
$\operatorname{osc}_{B_r}v<\infty$, then
\[
    |Dv(0)|\leq C\bigl(n,k,s,r,\operatorname{osc}_{B_r}v\bigr).
\]
\end{lemma}

\begin{proof}
We first establish the negative-direction ellipticity estimate
\begin{equation}\label{2.3}
    \kappa_i<0
    \quad\Longrightarrow\quad
    F_i\geq c(n,k)\sum_{j=1}^nF_j.
\end{equation}
Fix an index $i$ with $\kappa_i<0$. Necessarily $k<n$. Set
$a_i=\sigma_{k-1}(\kappa|i)$,
$b_i=\sigma_{k-2}(\kappa|i)$,
$c_i=\sigma_{k-3}(\kappa|i)$, and
$e_i=\sigma_k(\kappa|i)$.
Since $\kappa\in\Gamma_k$, the deletion vector $\kappa|i$ belongs
to $\Gamma_{k-1}$, so $a_i,b_i,c_i>0$. Moreover,
$e_i=\sigma_k(\kappa)-\kappa_i a_i>0$, and hence
$\kappa|i\in\Gamma_k$. Newton's inequalities applied to the
$(n-1)$-dimensional vector $\kappa|i$ give
\[
    e_ic_i\leq\vartheta a_ib_i,
    \qquad
    \vartheta=
    \frac{(n-k)(k-2)}{k(n-k+2)}<1.
\]
Since
$\sigma_k=e_i+\kappa_i a_i$ and
$\sigma_{k-2}=b_i+\kappa_i c_i$, we have
\[
    F^2
    =
    \frac{e_i+\kappa_i a_i}{b_i+\kappa_i c_i}
    \leq\frac{e_i}{b_i},
\]
where the last inequality follows from $\kappa_i<0$ and
$a_ib_i\geq e_ic_i$.

Differentiating
$F=\left(\frac{\sigma_k}{\sigma_{k-2}}\right)^{\frac12}$ with respect to $\kappa_i$ gives
\[
\begin{aligned}
    F_i
    &=
    \frac{a_i-F^2c_i}{2F\sigma_{k-2}}
    \geq
    \frac{(1-\vartheta)a_i}{2F\sigma_{k-2}},
    \\
    \sum_jF_j
    &=
    \frac{(n-k+1)\sigma_{k-1}
    -(n-k+3)F^2\sigma_{k-3}}
    {2F\sigma_{k-2}}
    \leq
    \frac{(n-k+1)a_i}{2F\sigma_{k-2}},
\end{aligned}
\]
where we used
$\sigma_{k-1}=a_i+\kappa_i b_i\leq a_i$.
This proves \eqref{2.3}. If $k=n$, then
$\Gamma_n$ is the positive cone and the assertion is vacuous.

Set $\varrho=\frac{3r}{4}$ and
$M=1+\operatorname{osc}_{B_r}v$. After a vertical translation,
we may assume $-M\leq v\leq0$ in $B_r$. Let
$W=\sqrt{1+|Dv|^2}$, $q=W^{-1}$, and define
\[
    \varphi
    =
    1-\frac{|x|^2}{\varrho^2}
    +\frac{v}{2M},
    \qquad
    \eta=e^{A\varphi}-1,
\]
where $A\geq1$ will be chosen below. Since
$\varphi\leq0$ on $\partial B_\varrho$ and
$\varphi(0)\geq\frac12$, the function $\eta W$ attains a positive
interior maximum on $\{\varphi>0\}\cap B_\varrho$.
All quantities below are evaluated at this maximum in a principal
orthonormal frame.

The first derivative equation gives
$\frac{q_i}{q}=\frac{\eta'}{\eta}\varphi_i$. Using the graph identity
$q_i=-\kappa_i v_i$, we obtain
\[
    \kappa_i v_i
    =
    -q\frac{\eta'}{\eta}\varphi_i.
\]
Hence $v_i\varphi_i\leq0$ whenever $\kappa_i\geq0$.

Let $R_i=\langle(x,0),e_i\rangle$ and
$\omega=\langle(x,0),N_v\rangle$. Since
$|\nabla v|^2=1-q^2$ and
$\sum_i v_iR_i=q\omega$, we have
\[
    \langle\nabla v,\nabla\varphi\rangle
    =
    \frac{1-q^2}{2M}
    -\frac{2q\omega}{\varrho^2}
    \geq
    \frac{1-q^2}{2M}
    -\frac{2q}{\varrho}.
\]
Suppose that
\[
    q<q_0:=
    \min\left\{\frac12,\frac{\varrho}{16M}\right\}.
\]
Then
$\langle\nabla v,\nabla\varphi\rangle\geq\frac1{4M}$.
If $k<n$, the directions with $\kappa_i\geq0$ contribute
nonpositively, and therefore
\[
\begin{aligned}
    \frac1{4M}
    \leq
    \sum_{\kappa_i<0}v_i\varphi_i \leq
    \left(\sum_{\kappa_i<0}v_i^2\right)^{\frac12}
    \left(\sum_{\kappa_i<0}\varphi_i^2\right)^{\frac12}
    \leq
    \left(\sum_{\kappa_i<0}\varphi_i^2\right)^{\frac12}.
\end{aligned}
\]
Together with \eqref{2.3}, this gives
\[
    |\nabla\varphi|_F^2
    \geq
    \frac{c(n,k)}{16M^2}\sum_iF_i.
\]
If $k=n$, then all $\kappa_i>0$, so
$v_i\varphi_i\leq0$ for every $i$, contradicting
$\langle\nabla v,\nabla\varphi\rangle\geq\frac1{4M}$.
Thus in this case $q\geq q_0$ at the maximum.

Assume now $k<n$. The graph identities give
\[
\begin{aligned}
    \Delta_F\varphi=
    -\frac{2}{\varrho^2}
      \sum_iF_i(1-v_i^2)
    +\frac{2s\omega}{\varrho^2}
    +\frac{sq}{2M}\geq
    -\frac{2}{\varrho^2}\sum_iF_i
    -\frac{2s}{\varrho}.
\end{aligned}
\]
By concavity and homogeneity,
$\sum_iF_i\geq F(I_n)>0$. Hence
\[
    \frac{2s}{\varrho}
    \leq
    \frac{2s}{\varrho F(I_n)}
    \sum_iF_i.
\]
Combining the preceding estimates, we obtain
\[
    \Delta_F\varphi
    +A|\nabla\varphi|_F^2
    \geq
    \left[
        -\frac{2}{\varrho^2}
        -\frac{2s}{\varrho F(I_n)}
        +\frac{Ac(n,k)}{16M^2}
    \right]
    \sum_iF_i.
\]
Thus $A=A(n,k,s,\varrho,M)$ can be chosen so that
\begin{equation}\label{gradient-positive}
    \Delta_F\varphi
    +A|\nabla\varphi|_F^2>0
\end{equation}
whenever $q<q_0$.

Next, from $\Delta_Fq=-q\mathcal E$, where
$\mathcal E=\sum_iF_i\kappa_i^2$, we obtain
\[
    \Delta_FW
    =
    W\mathcal E
    +\frac{2}{W}|\nabla W|_F^2.
\]
At the maximum of $\eta W$,
$\nabla W=-\frac{W}{\eta}\nabla\eta$. Therefore the two gradient terms
in $\Delta_F(\eta W)$ cancel, and
\[
    \Delta_F(\eta W)
    =
    W\Delta_F\eta+\eta W\mathcal E.
\]
Since
\[
    \Delta_F\eta
    =
    Ae^{A\varphi}
    \bigl(
        \Delta_F\varphi
        +A|\nabla\varphi|_F^2
    \bigr),
\]
\eqref{gradient-positive} implies $\Delta_F\eta>0$ under the
assumption $q<q_0$. Since $\mathcal E\geq0$, we obtain
$\Delta_F(\eta W)>0$, contradicting the fact that $\eta W$
has an interior maximum. Hence $q\geq q_0$ at the maximum in all
cases.

Finally, $\varphi(0)\geq\frac12$ and $\varphi\leq1$, so
\[
    \bigl(e^{\frac{A}{2}}-1\bigr)W(0)
    \leq
    \max(\eta W)
    \leq
    \frac{e^A-1}{q_0}.
\]
Thus $W(0)$, and hence $|Dv(0)|$, is bounded by a constant depending
only on $n,k,s,r$ and $\operatorname{osc}_{B_r}v$. This proves the
lemma.
\end{proof}

By Lemma~\ref{gradient}, there exists a constant
$K_1=K_1(n,k,R,K)>0$ such that
\[
    \|Dv\|_{L^\infty(B_{\frac{7R_0}{8}})}
    \leq K_1.
\]
Replacing $K$ by $\max\{K,K_1\}$ and retaining the same notation,
we shall henceforth assume that
\[
    |Dv|\leq K
    \qquad\text{in }B_{\frac{7R_0}{8}}.
\]
This enlargement depends only on the original data $n,k,R,K$.

We shall prove the following uniform separation estimate.

\begin{theorem}\label{thm:separation}
Under the above assumptions, there exists a constant
$c=c(n,k,R,K)>0$ such that
\[
    v(x)-u(x)\geq c
    \qquad\text{for all }x\in B_{\frac{R}{2}}.
\]
\end{theorem}

We divide the proof into five steps.

\subsection{The distance function}\label{distance-section}
Define the closed hypograph of $u$ by
\[
    E_u
    =
    \bigl\{
        (x,z)\in\mathbb{R}^{n+1}:
        x\in\overline{B}_{R_0},\ z\leq u(x)
    \bigr\},
\]
and, for $y\in B_{R_0}$, set
\[
    d(y)=\operatorname{dist}\bigl(X_v(y),E_u\bigr).
\]
Since $v(y)>u(y)$, we have $X_v(y)\notin E_u$. As $E_u$ is
closed, $d(y)>0$ for every $y\in B_{R_0}$. We write
$b=-\log d$.

We first compare the ambient distance $d$ quantitatively with the
gap $v-u$ and derive the Lipschitz estimate.

\begin{lemma}\label{lem:distance-gap}
For every $y\in B_{R_0}$,
\begin{equation}\label{3.2}
    \frac{v(y)-u(y)}{\sqrt{1+K^2}}
    \leq d(y)
    \leq v(y)-u(y).
\end{equation}
Moreover, $d$ is $\sqrt{1+K^2}$-Lipschitz continuous in
$B_{\frac{7R_0}{8}}$.
\end{lemma}

\begin{proof}
Since $(y,u(y))\in E_u$, the vertical competitor immediately gives
\[
    d(y)
    \leq
    \bigl|X_v(y)-(y,u(y))\bigr|
    =
    v(y)-u(y).
\]

To prove the opposite inequality, fix
$x\in\overline{B}_{R_0}$. Since $|Du|\leq K$,
\[
    |u(x)-u(y)|\leq K|x-y|.
\]
If $v(y)>u(x)$, then
\begin{align*}
    v(y)-u(y)
    &=
    v(y)-u(x)+u(x)-u(y)\\
    &\leq
    v(y)-u(x)+K|x-y|\\
    &\leq
    \sqrt{1+K^2}\,
    \sqrt{|x-y|^2+\bigl(v(y)-u(x)\bigr)^2}.
\end{align*}
The last square root is precisely the distance from $X_v(y)$ to
the vertical ray
\[
    \bigl\{(x,z):z\leq u(x)\bigr\}.
\]

If $v(y)\leq u(x)$, then
\[
    v(y)-u(y)
    \leq
    u(x)-u(y)
    \leq
    K|x-y|
    \leq
    \sqrt{1+K^2}\,|x-y|,
\]
while $|x-y|$ is the distance from $X_v(y)$ to the same vertical
ray. Thus, in either case,
\[
    v(y)-u(y)
    \leq
    \sqrt{1+K^2}\,
    \operatorname{dist}\bigl(
        X_v(y),\{(x,z):z\leq u(x)\}
    \bigr).
\]
Taking the infimum over $x\in\overline{B}_{R_0}$ yields
\[
    v(y)-u(y)
    \leq
    \sqrt{1+K^2}\,d(y),
\]
which proves \eqref{3.2}.

Finally, the distance to a closed set is $1$-Lipschitz in the
ambient Euclidean space. Hence, for
$y_1,y_2\in B_{\frac{7R_0}{8}}$,
\begin{align*}
    |d(y_1)-d(y_2)|
    &\leq
    |X_v(y_1)-X_v(y_2)|\\
    &=
    \sqrt{
        |y_1-y_2|^2
        +
        |v(y_1)-v(y_2)|^2
    }\\
    &\leq
    \sqrt{1+K^2}\,|y_1-y_2|,
\end{align*}
where the last inequality follows from
$\|Dv\|_{L^\infty(B_{\frac{7R_0}{8}})}\leq K$.
\end{proof}

\begin{remark}\label{rem:vertical-gap-hessians}
The pointwise comparison \eqref{3.2} does not imply a corresponding
comparison of the covariant Hessians. Indeed, set $\rho=v-u$ and
denote by $\nabla^u$ and $\nabla^v$ the Levi--Civita connections
of $\Sigma_u$ and $\Sigma_v$, respectively. For a graph function
$w$, the Christoffel symbols in graph coordinates are
\[
    \Gamma_{ij}^{\ell}[w]
    =
    \frac{D_\ell w\,D_{ij}w}{1+|Dw|^2}.
\]
Therefore, at a common base point,
\begin{align*}
    \bigl(\nabla_u^2\rho\bigr)_{ij}
    &=
    D_{ij}v
    -
    \frac{1+Du\cdot Dv}{1+|Du|^2}\,D_{ij}u,\\
    \bigl(\nabla_v^2\rho\bigr)_{ij}
    &=
    \frac{1+Du\cdot Dv}{1+|Dv|^2}\,D_{ij}v
    -
    D_{ij}u.
\end{align*}
Thus the two intrinsic Hessians contain different
gradient-dependent coefficients. In particular, a differential
inequality for the vertical gap does not follow from
\eqref{3.2}. We therefore derive the required differential
inequality directly for the geometric distance $d$ by a
two-point comparison.
\end{remark}

\subsection{The compression inequality}\label{compression-section}
The two-point second-variation calculation leads to the matrix
$M(I_n-M)^{-1}$. The following lemma gives the compression estimate
needed below.

\begin{lemma}\label{compression}
Let $3\leq k\leq n$, and let $M\in\Gamma_k$ be a real symmetric
matrix such that $I_n-M>0$. Set
\[
    D=M(I_n-M)^{-1},
    \qquad
    a_0=\frac{n-k}{n}.
\]
For $a_0<a<1$, define
\[
    \theta^2
    =
    \frac{a-a_0}{\frac{n}{k}+a},
    \qquad
    \theta>0.
\]
In any orthonormal basis $\{e_1,\ldots,e_n\}$, let
\[
    P=\operatorname{diag}(c,1,\ldots,1),
    \qquad
    B=PDP+a(1-c^2)e_1\otimes e_1,
    \qquad
    -1\leq c\leq1.
\]
Then $B\in\Gamma_k$ and
\[
    F(B)\geq\theta F(M).
\]
\end{lemma}

\begin{proof}
Since
\[
    D-M=M^2(I_n-M)^{-1}\geq0,
\]
and the G{\aa}rding cone $\Gamma_k$ is monotone under
positive-semidefinite perturbations, we have
\[
    D\in\Gamma_k.
\]
The ellipticity of $F$ on $\Gamma_k$ then gives
\[
    F(D)\geq F(M).
\]

Fix the orthonormal basis from the statement and set
\[
    D'=(D_{\alpha\beta})_{2\leq\alpha,\beta\leq n}.
\]
Choose an orthonormal eigenbasis $\{\xi_j\}_{j=1}^n$ of $M$ and write
\[
    M\xi_j=\mu_j\xi_j,
    \qquad
    \tau_j=\frac{\mu_j}{1-\mu_j},
    \qquad
    w_j=\langle e_1,\xi_j\rangle^2.
\]
Then $\mu_j<1$, $D\xi_j=\tau_j\xi_j$, $w_j\geq0$, and
$\sum_jw_j=1$. We write $\mu|j$ and $\tau|j$ for the vectors
obtained by deleting the $j$th entry, and use the convention
$\sigma_r=0$ whenever $r$ exceeds the dimension of the vector or
matrix under consideration.

We first prove
\[
    \sigma_k(\tau|j)
    \geq
    -a_0\sigma_{k-1}(\tau|j),
    \qquad
    1\leq j\leq n.
\]
Fix $j$ and suppose first that $k<n$. Set
$x=\mu|j\in\Gamma_{k-1}$. For $0\leq t\leq1$, define
\[
    x_i(t)=\frac{x_i}{1-tx_i},
    \qquad
    S_r(t)=\sigma_r(x(t)),
    \qquad
    R(t)=\frac{S_k(t)}{S_{k-1}(t)}.
\]
Since $x_i'(t)=x_i(t)^2\geq0$, the path $x(t)$ remains in
$\Gamma_{k-1}$ and $S_{k-1}(t)>0$. Moreover,
\[
    \sigma_k(\mu)
    =
    \sigma_k(x)+\mu_j\sigma_{k-1}(x)>0,
\]
so
\[
    R(0)
    =
    \frac{\sigma_k(x)}{\sigma_{k-1}(x)}
    >-\mu_j>-1.
\]

Differentiating $S_r$ gives
\[
    S_r'
    =
    \sum_i x_i(t)^2\sigma_{r-1}(x(t)|i)
    =
    S_1S_r-(r+1)S_{r+1}.
\]
Hence
\[
    R'
    =
    kR^2-(k+1)\frac{S_{k+1}}{S_{k-1}}.
\]
If $k\leq n-2$, the normalized Newton inequality in $n-1$
variables gives
\[
    \frac{S_{k+1}}{S_{k-1}}
    \leq
    \frac{
        \binom{n-1}{k-1}\binom{n-1}{k+1}
    }{
        \binom{n-1}{k}^2
    }R^2
    =
    \frac{k(n-k-1)}
         {(k+1)(n-k)}R^2.
\]
When $k=n-1$, we have $S_{k+1}=0$. Therefore
\begin{equation}\label{3.3}
    R'
    \geq
    \frac{k}{n-k}R^2.
\end{equation}
In particular, $R$ is nondecreasing. If $R$ remains negative on
$[0,1]$, then
\[
    \left(-\frac1R\right)'
    =
    \frac{R'}{R^2}
    \geq
    \frac{k}{n-k}.
\]
Integrating from $0$ to $1$, we obtain
\[
    R(1)
    \geq
    \frac{R(0)}
         {1-\frac{k}{n-k}R(0)}
    >
    -\frac{n-k}{n}
    =
    -a_0.
\]
If $R$ becomes nonnegative before $t=1$, then $R(1)\geq0$.
Since $x(1)=\tau|j$, it follows that
\[
    \sigma_k(\tau|j)
    \geq
    -a_0\sigma_{k-1}(\tau|j).
\]
If $k=n$, then $a_0=0$ and $\sigma_n(\tau|j)=0$, so the same
inequality is immediate.

We next pass to the principal submatrix $D'$. The $(1,1)$-cofactor
of $zI_n-D$ gives
\[
    \det(zI_{n-1}-D')
    =
    \langle
        \operatorname{adj}(zI_n-D)e_1,e_1
    \rangle
    =
    \sum_{j=1}^n
    w_j\prod_{i\neq j}(z-\tau_i).
\]
Comparing coefficients yields
\begin{equation}\label{3.4}
    \sigma_r(D')
    =
    \sum_{j=1}^n
    w_j\sigma_r(\tau|j),
    \qquad
    0\leq r\leq n-1.
\end{equation}
The identity also holds for $r=n$, with both sides equal to zero.
Since each $\tau|j$ belongs to $\Gamma_{k-1}$,
\[
    D'\in\Gamma_{k-1},
    \qquad
    \sigma_k(D')
    \geq
    -a_0\sigma_{k-1}(D').
\]

Set
\[
    B_0=\operatorname{diag}(a,D').
\]
For $1\leq r\leq k-1$,
\[
    \sigma_r(B_0)
    =
    \sigma_r(D')
    +
    a\sigma_{r-1}(D')
    >0,
\]
while
\begin{equation}\label{3.5}
    \sigma_k(B_0)
    =
    \sigma_k(D')
    +
    a\sigma_{k-1}(D')
    \geq
    (a-a_0)\sigma_{k-1}(D')
    >0.
\end{equation}
Thus $B_0\in\Gamma_k$.

We now estimate $F(B_0)$. Set
\[
    q_*=F(M)^2.
\]
Fix $j$ and write $x=\mu|j$. Replacing the $j$th eigenvalue
$\mu_j$ by any $t\geq\mu_j$ preserves admissibility. Evaluating
the quotient at $t=1$ and then letting $t\to+\infty$, we obtain
\begin{align}
    q_*
    &\leq
    \frac{
        \sigma_k(x)+\sigma_{k-1}(x)
    }{
        \sigma_{k-2}(x)+\sigma_{k-3}(x)
    }
    \leq
    \frac{n}{k}
    \frac{\sigma_{k-1}(x)}
         {\sigma_{k-2}(x)},
    \label{3.6}
    \\
    q_*
    &\leq
    \frac{\sigma_{k-1}(x)}
         {\sigma_{k-3}(x)}.
    \label{3.7}
\end{align}
For the second inequality in \eqref{3.6}, it is enough to show
\[
    \sigma_k(x)
    \leq
    \frac{n-k}{k}\sigma_{k-1}(x).
\]
This is immediate if $\sigma_k(x)\leq0$. If
$\sigma_k(x)>0$, then $x\in\Gamma_k$, and since each component
of $x$ is less than $1$,
\[
    \frac{\sigma_k(x)}
         {\sigma_{k-1}(x)}
    \leq
    \frac{\sigma_k(1,\ldots,1)}
         {\sigma_{k-1}(1,\ldots,1)}
    =
    \frac{n-k}{k}.
\]
For $k=n$, we have $\sigma_k(x)=0$.

Since
\[
    \tau_i-\mu_i
    =
    \frac{\mu_i^2}{1-\mu_i}
    \geq0,
\]
the estimates \eqref{3.6} and \eqref{3.7} remain valid after
replacing $\mu|j$ by $\tau|j$. Hence
\[
    q_*
    \bigl(
        \sigma_{k-2}(\tau|j)
        +
        a\sigma_{k-3}(\tau|j)
    \bigr)
    \leq
    \left(
        \frac{n}{k}+a
    \right)
    \sigma_{k-1}(\tau|j).
\]
Multiplying by $w_j$, summing over $j$, and using \eqref{3.4},
we obtain
\[
    q_*
    \bigl(
        \sigma_{k-2}(D')
        +
        a\sigma_{k-3}(D')
    \bigr)
    \leq
    \left(
        \frac{n}{k}+a
    \right)
    \sigma_{k-1}(D').
\]
Together with \eqref{3.5}, this yields
\begin{align*}
    F(B_0)^2
    &=
    \frac{
        \sigma_k(D')+a\sigma_{k-1}(D')
    }{
        \sigma_{k-2}(D')+a\sigma_{k-3}(D')
    }\\
    &\geq
    \frac{
        (a-a_0)\sigma_{k-1}(D')
    }{
        \sigma_{k-2}(D')+a\sigma_{k-3}(D')
    }\\
    &\geq
    \frac{a-a_0}{\frac{n}{k}+a}\,q_*
    =
    \theta^2q_*.
\end{align*}
Therefore
\[
    F(B_0)\geq\theta F(M).
\]

Finally, let $J$ be an index set containing $1$. Expanding the
corresponding principal minor gives
\[
    \det B[J]
    =
    c^2\det D[J]
    +
    a(1-c^2)\det D[J\setminus\{1\}].
\]
The principal minors not containing $1$ are unchanged. Summing over
all index sets of size $r$ gives
\[
    \sigma_r(B)
    =
    c^2\sigma_r(D)
    +
    (1-c^2)\sigma_r(B_0),
    \qquad
    1\leq r\leq n.
\]
Thus $B\in\Gamma_k$.

Since $0<\theta<1$, $F(D)\geq F(M)$, and
$F(B_0)\geq\theta F(M)$,
\begin{align*}
    \sigma_k(B)
    &=
    c^2\sigma_k(D)
    +
    (1-c^2)\sigma_k(B_0)\\
    &\geq
    \theta^2q_*
    \left[
        c^2\sigma_{k-2}(D)
        +
        (1-c^2)\sigma_{k-2}(B_0)
    \right]=
    \theta^2q_*\sigma_{k-2}(B).
\end{align*}
Since $\sigma_{k-2}(B)>0$,
\[
    F(B)^2
    \geq
    \theta^2F(M)^2,
\]
and hence
\[
    F(B)\geq\theta F(M).
\]
\end{proof}

From now on, we fix
\begin{equation}\label{3.8}
    a=1-\frac{k}{2n},
    \qquad
    \theta=
    \frac{k}{\sqrt{2n^2+2nk-k^2}},
    \qquad
    s_0=\frac{\theta}{4}.
\end{equation}
Since $a-a_0=\frac{k}{2n}$,
\[
    \theta^2
    =
    \frac{
        \frac{k}{2n}
    }{
        \frac{n}{k}+1-\frac{k}{2n}
    }
    =
    \frac{k^2}{2n^2+2nk-k^2},
\]
which agrees with the definition of $\theta$ in
Lemma~\ref{compression}. In particular, $0<s_0<\frac{\theta}{2}$.

\subsection{The two-surface inequality}
We now return to the distance function $d$. Since $d$ need not be
smooth, we formulate its differential inequality using a smooth test
function at an interior closest pair. Lemma~\ref{compression}
provides the matrix estimate needed in the second-variation
calculation.

\begin{proposition}\label{prop:distance}
Suppose that $u$ and $v$ are admissible solutions of
\[
    F(\kappa[u])=1,
    \qquad
    F(\kappa[v])=s,
    \qquad
    0<s\leq\frac{\theta}{2},
\]
where $a$ and $\theta$ are given by \eqref{3.8}.
Let $X=X_u(x)\in\Sigma_u$ and $Y=X_v(y)\in\Sigma_v$ be interior
points such that
\[
    |Y-X|=\operatorname{dist}(Y,E_u)=d>0.
\]
Suppose that $\phi$ is $C^2$ in a neighborhood of $Y$ and
\[
    (X',Y')
    \longmapsto
    \phi(Y')+\log|Y'-X'|
\]
has a local minimum at $(X,Y)$. Then
\begin{equation}\label{3.9}
    \Delta_F\phi
    \geq
    \gamma|\nabla\phi|_F^2+\frac{\delta}{d},
    \qquad
    \gamma=1-a,
    \qquad
    \delta=\frac{\theta}{2}.
\end{equation}
All quantities in \eqref{3.9} are evaluated at $Y$.
\end{proposition}

\begin{proof}
Choose geodesic coordinates centered at $X$ and $Y$, and write
$X_\alpha,Y_i$ for the corresponding coordinate vectors. Their
orthonormal bases will be chosen below. All quantities in the
following calculation are evaluated at $X$ and $Y$. Thus
\[
    \langle X_\alpha,X_\beta\rangle=\delta_{\alpha\beta},
    \qquad
    \langle Y_i,Y_j\rangle=\delta_{ij},
    \qquad
    X_{\alpha\beta}=-h^u_{\alpha\beta}N_u,
    \qquad
    Y_{ij}=-h^v_{ij}N_v.
\]

Set
\[
    \Psi(X',Y')
    =
    \phi(Y')+\log|Y'-X'|.
\]
Differentiating in the lower variables gives
\[
    0=\Psi_\alpha
    =
    -\frac{\langle Y-X,X_\alpha\rangle}{d^2}.
\]
Hence $Y-X$ is normal to $\Sigma_u$. Since $X$ is an interior
nearest point on the boundary of $E_u$ and $N_u$ is the downward
unit normal,
\[
    Y-X=-d N_u.
\]

Write
\[
    c=\langle N_u,N_v\rangle,
    \qquad
    p_{\alpha i}=\langle X_\alpha,Y_i\rangle,
    \qquad
    \xi_i=\langle N_u,Y_i\rangle.
\]
Differentiating $\Psi$ in the upper variables gives
\[
    0=\Psi_i
    =
    \phi_i+\frac{\langle Y-X,Y_i\rangle}{d^2},
\]
and therefore
\begin{equation}\label{eq:first-distance}
    \phi_i=\frac{\xi_i}{d}.
\end{equation}

The second derivatives are
\[
    \Psi_{\alpha\beta}
    =
    \frac{\delta_{\alpha\beta}-dh^u_{\alpha\beta}}{d^2},
    \qquad
    \Psi_{\alpha j}
    =
    -\frac{p_{\alpha j}}{d^2}.
\]
Differentiating twice in the upper variables gives
\[
    \Psi_{ij}
    =
    \nabla_i\nabla_j\phi
    +\frac{c}{d}h^v_{ij}
    +\frac{\delta_{ij}-2\xi_i\xi_j}{d^2}.
\]
Hence
\begin{equation}\label{eq:block-distance}
    \begin{pmatrix}
        \displaystyle
        d^{-2}(\delta_{\alpha\beta}-dh^u_{\alpha\beta})
        &
        \displaystyle
        -d^{-2}p_{\alpha j}
        \\[2mm]
        \displaystyle
        -d^{-2}p_{\beta i}
        &
        \displaystyle
        \nabla_i\nabla_j\phi
        +d^{-1}ch^v_{ij}
        +d^{-2}(\delta_{ij}-2\xi_i\xi_j)
    \end{pmatrix}
    \geq0.
\end{equation}

Since $N_u$ and $N_v$ are both downward graph normals, $c>-1$.
If $|c|<1$, choose
\[
    X_1
    =
    \frac{N_v-cN_u}{\sqrt{1-c^2}},
    \qquad
    Y_1
    =
    \frac{cN_v-N_u}{\sqrt{1-c^2}},
\]
and complete these vectors by a common orthonormal basis
$X_\alpha=Y_\alpha$, $2\leq\alpha\leq n$, of the common tangent
subspace. If $c=1$, take the same orthonormal basis in the two
tangent spaces. In either case,
\begin{equation}\label{eq:adapted-distance}
\begin{gathered}
    P=(p_{\alpha i})
    =
    \operatorname{diag}(c,1,\ldots,1),\\
    \xi=(-\sqrt{1-c^2},0,\ldots,0),
    \qquad
    \xi\otimes\xi
    =
    (1-c^2)e_1\otimes e_1,\\
    P^{\mathsf T}P=P^2=I_n-\xi\otimes\xi.
\end{gathered}
\end{equation}

Set
\[
    A_u=(h^u_{\alpha\beta}),
    \qquad
    A_v=(h^v_{ij}),
    \qquad
    M=dA_u.
\]
Multiplying \eqref{eq:block-distance} by $d^2$ gives
\[
    \begin{pmatrix}
        I_n-M & -P\\
        -P &
        d^2\nabla^2\phi
        +cdA_v
        +I_n-2\xi\otimes\xi
    \end{pmatrix}
    \geq0.
\]
Hence
\[
    I_n-M\geq0.
\]
Since $A_u\in\Gamma_k$ and $d>0$,
\[
    M\in\Gamma_k,
    \qquad
    F(M)=dF(A_u)=d.
\]

To regularize the possibly singular block $I_n-M$, let
$\varepsilon>0$ and set
\[
    t=1+\varepsilon,
    \qquad
    M_t=\frac{M}{t},
    \qquad
    D_t=M(tI_n-M)^{-1}.
\]
Then
\[
    M_t\in\Gamma_k,
    \qquad
    I_n-M_t>0,
    \qquad
    D_t=M_t(I_n-M_t)^{-1}.
\]
Adding $\varepsilon I_n$ to the upper-left block of the preceding
matrix preserves nonnegativity. Since
\[
    tI_n-M=(I_n-M)+\varepsilon I_n>0,
\]
the Schur complement gives
\[
    d^2\nabla^2\phi
    +cdA_v
    +I_n-2\xi\otimes\xi
    \geq
    P(tI_n-M)^{-1}P.
\]
Using
\[
    (tI_n-M)^{-1}
    =
    \frac1t(I_n+D_t)
\]
and \eqref{eq:adapted-distance}, we obtain
\[
    \nabla^2\phi
    \geq
    -\frac{c}{d}A_v
    +\frac{PD_tP}{td^2}
    +\left(2-\frac1t\right)
       \frac{\xi\otimes\xi}{d^2}
    -\frac{\varepsilon}{td^2}I_n.
\]

Define
\[
    B_t=PD_tP+a\xi\otimes\xi.
\]
Lemma~\ref{compression}, applied to $M_t$, yields
\[
    B_t\in\Gamma_k,
    \qquad
    F(B_t)
    \geq
    \theta F(M_t)
    =
    \frac{\theta d}{t}.
\]
Therefore
\begin{equation}\label{eq:hessian-distance}
    \nabla^2\phi
    \geq
    -\frac{c}{d}A_v
    +\frac{B_t}{td^2}
    +\left(2-\frac{1+a}{t}\right)
       \frac{\xi\otimes\xi}{d^2}
    -\frac{\varepsilon}{td^2}I_n.
\end{equation}

All coefficients $F^{ij}$ below are evaluated at $A_v$.
By homogeneity,
\[
    F^{ij}h^v_{ij}
    =
    F(A_v)
    =
    s.
\]
By concavity,
\[
\begin{aligned}
    F(B_t)
    \leq
    F(A_v)
    +
    F^{ij}\bigl((B_t)_{ij}-h^v_{ij}\bigr)=
    F^{ij}(B_t)_{ij}.
\end{aligned}
\]
Moreover, \eqref{eq:first-distance} gives
\[
    \frac{F^{ij}\xi_i\xi_j}{d^2}
    =
    F^{ij}\phi_i\phi_j
    =
    |\nabla\phi|_F^2.
\]

Contracting \eqref{eq:hessian-distance} with $(F^{ij})$ gives
\begin{align*}
    \Delta_F\phi
    &\geq
    \left(2-\frac{1+a}{t}\right)
    |\nabla\phi|_F^2
    +
    \frac{F^{ij}(B_t)_{ij}}{td^2}
    -
    \frac{cs}{d}
    -
    \frac{\varepsilon}{td^2}\sum_iF^{ii}\\
    &\geq
    \left(2-\frac{1+a}{t}\right)
    |\nabla\phi|_F^2
    +
    \frac{\theta}{t^2d}
    -
    \frac{cs}{d}
    -
    \frac{\varepsilon}{td^2}\sum_iF^{ii}.
\end{align*}
At the fixed contact pair, $d>0$ and $F^{ij}(A_v)$ are independent
of $\varepsilon$. Letting $\varepsilon\downarrow0$, we obtain
\[
    \Delta_F\phi
    \geq
    (1-a)|\nabla\phi|_F^2
    +
    \frac{\theta-cs}{d}.
\]
Since $c\leq1$ and $s\leq\frac{\theta}{2}$,
\[
    \theta-cs
    \geq
    \theta-s
    \geq
    \frac{\theta}{2}.
\]
Thus
\[
    \Delta_F\phi
    \geq
    \gamma|\nabla\phi|_F^2
    +
    \frac{\delta}{d},
\]
where
\[
    \gamma=1-a,
    \qquad
    \delta=\frac{\theta}{2}.
\]
This proves \eqref{3.9}.

If $\phi$ is an upper test function for $b=-\log d$ at $Y$ and
$Y$ has several nearest points on the lower graph, the same
coordinate calculation applies separately to every interior
closest pair.
\end{proof}

\subsection{A positive lower bound at one point}\label{seed-section}
From now on, we take $s=s_0$, where $s_0$ is fixed in
\eqref{3.8}. On the upper graph $\Sigma_v$, define
\[
    T_v^{ij}
    =
    T_{k-1}^{ij}-s^2T_{k-3}^{ij}
    =
    2s\sigma_{k-2}F^{ij}.
\]
All geometric quantities below are computed with respect to the
induced metric on $\Sigma_v$, and we write
$\sigma_j=\sigma_j(\kappa[v])$.

Set
\[
    \Delta_{T_v}\phi=T_v^{ij}\nabla^2_{ij}\phi,
    \qquad
    |\nabla\phi|_{T_v}^2=T_v^{ij}\phi_i\phi_j.
\]
By ellipticity, $(T_v^{ij})$ is positive definite. Since $s$ is
constant and the Newton tensors are divergence-free,
$\nabla_iT_v^{ij}=0$. Moreover, the Newton identities give
\begin{equation}\label{3.10}
    T_v^{ij}g_{ij}
    \leq (n-k+1)\sigma_{k-1},
    \qquad
    T_v^{ij}h_{ij}
    =2s^2\sigma_{k-2}.
\end{equation}

On the level set $F=s$, we have
$\sigma_k=s^2\sigma_{k-2}$. The Newton--Maclaurin inequalities
therefore give
\[
    \sigma_{k-2}\geq c_*(n,k,s)>0.
\]
Multiplying \eqref{3.9} by $2s\sigma_{k-2}$ gives
\[
    \Delta_{T_v}b
    \geq
    \gamma|\nabla b|_{T_v}^2
    +\frac{2s\sigma_{k-2}\delta}{d}.
\]
Set
\[
    \delta_0
    :=
    2s\,c_*(n,k,s)\delta>0.
\]
Since $s=s_0$ depends only on $n$ and $k$,
$\delta_0=\delta_0(n,k)>0$. Hence, in the upper-support sense,
\begin{equation}\label{3.11}
    \Delta_{T_v}b
    \geq
    \gamma|\nabla b|_{T_v}^2
    +\frac{\delta_0}{d}.
\end{equation}

We next establish the local curvature integral bounds needed below.
Suppose that $|Dv|\leq K$ on $B_{r_2}$, and let
$0\leq\zeta\in C_c^\infty(B_{r_2})$, identified with its lift to
$\Sigma_v$. For $1\leq j\leq k$, the divergence-free property of
$T_{j-1}$ and the identity
$T_{j-1}^{ab}v_{ab}=\frac{j\sigma_j}{W_v}$ give
\[
\begin{aligned}
    j\int_{\Sigma_v}\frac{\zeta\sigma_j}{W_v}\,d\mu=
    -\int_{\Sigma_v}T_{j-1}^{ab}\zeta_a v_b\,d\mu\leq
    C(n)\int_{\Sigma_v}|\nabla\zeta|\sigma_{j-1}\,d\mu.
\end{aligned}
\]
Here we used the positivity of $T_{j-1}$,
$T_{j-1}^{ab}g_{ab}=(n-j+1)\sigma_{j-1}$, and
$|\nabla v|\leq1$. Since
\[
    \frac1{W_v}\geq\frac1{\sqrt{1+K^2}},
    \qquad
    d\mu=W_v\,dx,
\]
an iteration with nested cutoffs, starting from the area estimate
for $j=0$, yields
\begin{equation}\label{3.12}
    \int_{X_v(B_{r_1})}\sigma_j\,d\mu
    \leq C(n,j,K,r_1,r_2),
    \qquad
    0<r_1<r_2,\quad 1\leq j\leq k.
\end{equation}

Set
\[
    R_0=\frac{3R}{4},
    \qquad
    \Omega_v=X_v\left(B_{\frac{R_0}{16}}\right).
\]
The interior gradient bound for $v$ is available on
$B_{\frac{R_0}{8}}$. If
\[
    \max_{\overline{B}_{\frac{R_0}{16}}}d
    \geq\frac{R_0}{2},
\]
the required pointwise lower bound is immediate. We may therefore
assume
\[
    \max_{\overline{B}_{\frac{R_0}{16}}}d
    <\frac{R_0}{2}.
\]
Under this assumption, every nearest point in $E_u$ associated with
$y\in\overline{B}_{\frac{R_0}{16}}$ lies on the graph of $u$ over
an interior base point, since the horizontal distance to the
lateral boundary is at least $\frac{15R_0}{16}$.

On $\Omega_v$, consider
\[
    \begin{cases}
        -\Delta_{T_v}Q=1
            & \text{in }\Omega_v,\\
        Q=0
            & \text{on }\partial\Omega_v.
    \end{cases}
\]
For each fixed smooth graph, $T_v$ is uniformly elliptic on
$\overline{\Omega_v}$. Thus the problem has a classical solution,
and $Q>0$ in $\Omega_v$. No uniform bound for the ellipticity ratio
will be used.

Choose
\[
    0\leq\zeta\in C_c^\infty\left(B_{\frac{R_0}{16}}\right),
    \qquad
    \zeta=1
    \quad\text{on }B_{\frac{R_0}{32}},
\]
and lift $\zeta$ to $\Omega_v$. Since $\nabla_iT_v^{ij}=0$,
integration by parts gives
\[
    \int_{\Omega_v}\zeta\,d\mu
    =
    -\int_{\Omega_v}Q\,\Delta_{T_v}\zeta\,d\mu.
\]
Hence
\[
\begin{aligned}
    \left|B_{\frac{R_0}{32}}\right|
    \leq
    \int_{\Omega_v}\zeta\,d\mu\leq
    \left(\max_{\overline{\Omega_v}}Q\right)
    \int_{\Omega_v}|\Delta_{T_v}\zeta|\,d\mu.
\end{aligned}
\]
For a cutoff depending only on the base variables,
\[
    |\Delta_{T_v}\zeta|
    \leq
    |D^2\zeta|\,T_v^{ij}g_{ij}
    +
    |D\zeta|\,|T_v^{ij}h_{ij}|.
\]
By \eqref{3.10}, \eqref{3.12}, and the area bound,
\[
    \int_{\Omega_v}|\Delta_{T_v}\zeta|\,d\mu\leq C.
\]
Therefore
\[
    \max_{\overline{\Omega_v}}Q\geq c>0,
\]
where the constant is independent of any ellipticity-ratio bound.

Set
\[
    M_d=\max_{\overline{\Omega_v}}d,
    \qquad
    p=\frac{1}{2}\min\{\gamma,1\},
    \qquad
    \alpha_0=\frac{p\delta_0}{2}M_d^{p-1}.
\]
Since $0<p<1$, \eqref{3.11} implies, in the corresponding
lower-support sense,
\[
\begin{aligned}
    \Delta_{T_v}(d^p)
    &\leq
    p(p-\gamma)d^p|\nabla b|_{T_v}^2
    -p\delta_0d^{p-1}\\
    &\leq
    -p\delta_0d^{p-1}\\
    &\leq
    -p\delta_0M_d^{p-1}
    =
    -2\alpha_0.
\end{aligned}
\]
We claim that
\[
    d^p\geq\alpha_0Q
    \qquad\text{in }\Omega_v.
\]

Suppose otherwise. Since $Q=0$ on $\partial\Omega_v$,
$d^p-\alpha_0Q$ attains a negative minimum at some
$P_*\in\Omega_v$. Set
\[
    m=d(P_*)^p-\alpha_0Q(P_*)<0,
    \qquad
    \psi=\alpha_0Q+m.
\]
Then $\psi\leq d^p$ and
$\psi(P_*)=d(P_*)^p>0$. Hence $\psi>0$ near $P_*$ and
\[
    \phi=-\frac1p\log\psi
\]
is a smooth upper support for $b=-\log d$ at $P_*$. Applying
\eqref{3.11} to $\phi$, and using
$\psi(P_*)=d(P_*)^p$, gives
\[
\begin{aligned}
    \Delta_{T_v}\psi
    &=
    -p\psi\,\Delta_{T_v}\phi
    +p^2\psi|\nabla\phi|_{T_v}^2\\
    &\leq
    p(p-\gamma)\psi|\nabla\phi|_{T_v}^2
    -p\delta_0\frac{\psi}{d}\\
    &\leq
    -p\delta_0d^{p-1}\leq
    -2\alpha_0.
\end{aligned}
\]
On the other hand,
\[
    \Delta_{T_v}\psi
    =
    \alpha_0\Delta_{T_v}Q
    =
    -\alpha_0,
\]
a contradiction. Thus the claim holds.

At a maximum point of $Q$,
\[
    M_d^p
    \geq
    \alpha_0\max_{\overline{\Omega_v}}Q
    =
    \frac{p\delta_0}{2}
    M_d^{p-1}\max_{\overline{\Omega_v}}Q.
\]
Therefore
\[
    M_d
    \geq
    \frac{p\delta_0}{2}
    \max_{\overline{\Omega_v}}Q
    \geq c_0>0.
\]

Choose $y_0\in\overline{B}_{\frac{R_0}{16}}$ with
$d(y_0)=M_d$. By Lemma~\ref{lem:distance-gap}, $d$ is
$L$-Lipschitz on the relevant interior ball, where
$L=\sqrt{1+K^2}$. Set
\[
    d_*=\frac{c_0}{2},
    \qquad
    r_*=
    \min\left\{
        \frac{R_0}{64},
        \frac{c_0}{2L},
        1
    \right\}.
\]
Then
\[
    d(y)\geq d_*
    \qquad
    \text{for }y\in B_{r_*}(y_0).
\]
Both $d_*$ and $r_*$ have positive lower bounds depending only on
$n,k,R,K$.

\subsection{The doubling argument}\label{doubling-section}

\subsubsection*{Normalized linearization.}

We establish the directional bounds for the normalized
linearization needed in the propagation argument.

\begin{lemma}\label{normalized-linearization}
Let $3\leq k\leq n$ and let $\lambda\in\Gamma_k$ satisfy
$\sigma_k(\lambda)=\sigma_{k-2}(\lambda)$. Set
\[
    \widehat\lambda=(\lambda,1)\in\mathbb{R}^{n+1},
    \qquad
    \widehat F=\frac{\sigma_k}{\sigma_{k-1}}.
\]
Then $\widehat\lambda\in\Gamma_k$ and
$\widehat F(\widehat\lambda)=1$. For $1\leq i\leq n$, the coefficients
\begin{equation}\label{3.13}
    a_i=\widehat F_i(\widehat\lambda)
    =
    \frac{\sigma_{k-1}(\lambda|i)-\sigma_{k-3}(\lambda|i)}
         {\sigma_{k-1}(\lambda)+\sigma_{k-2}(\lambda)}
\end{equation}
are positive and satisfy
\begin{equation}\label{3.14}
\begin{gathered}
    \sum_{i=1}^n a_i\leq n-k+2,
    \qquad
    a_i(1+|\lambda_i|)^2\geq c(n,k),
    \\
    \lambda_i\leq0
    \quad\Longrightarrow\quad
    a_i\geq c(n,k).
\end{gathered}
\end{equation}
\end{lemma}

\begin{proof}
For $1\leq r\leq k$,
\[
    \sigma_r(\widehat\lambda)
    =
    \sigma_r(\lambda)+\sigma_{r-1}(\lambda)>0.
\]
Thus $\widehat\lambda\in\Gamma_k$, and
\[
    \widehat F(\widehat\lambda)
    =
    \frac{\sigma_k(\lambda)+\sigma_{k-1}(\lambda)}
         {\sigma_{k-1}(\lambda)+\sigma_{k-2}(\lambda)}
    =1.
\]
Differentiating gives
\[
\begin{aligned}
    \widehat F_i(\widehat\lambda)
    =
    \frac{\sigma_{k-1}(\widehat\lambda|i)
          -\sigma_{k-2}(\widehat\lambda|i)}
         {\sigma_{k-1}(\widehat\lambda)}
    =
    \frac{\sigma_{k-1}(\lambda|i)-\sigma_{k-3}(\lambda|i)}
         {\sigma_{k-1}(\lambda)+\sigma_{k-2}(\lambda)},
    \qquad
    1\leq i\leq n,
\end{aligned}
\]
which proves \eqref{3.13}.

To estimate these derivatives, fix $1\leq\alpha\leq n+1$ and set
\[
    A=\sigma_k(\widehat\lambda|\alpha),
    \qquad
    B=\sigma_{k-1}(\widehat\lambda|\alpha),
    \qquad
    C=\sigma_{k-2}(\widehat\lambda|\alpha),
    \qquad
    \tau=\widehat\lambda_\alpha.
\]
Since $\widehat\lambda|\alpha\in\Gamma_{k-1}$, we have $B,C>0$.
Newton's inequality in dimension $n$ gives
\[
    AC
    \leq
    \frac{\binom{n}{k}\binom{n}{k-2}}
         {\binom{n}{k-1}^2}B^2
    =
    \frac{(k-1)(n-k+1)}
         {k(n-k+2)}B^2.
\]
If $A\leq0$, the same estimate is immediate. Hence
\[
    B^2-AC\geq c_0B^2,
    \qquad
    c_0=\frac{n+1}{k(n-k+2)}>0.
\]
Using
\[
    \sigma_k(\widehat\lambda)=A+\tau B,
    \qquad
    \sigma_{k-1}(\widehat\lambda)=B+\tau C,
\]
we obtain
\[
    \widehat F_\alpha(\widehat\lambda)
    =
    \frac{B^2-AC}{(B+\tau C)^2}
    \geq
    c_0
    \left(
        \frac{B}{B+\tau C}
    \right)^2
    >0.
\]
Since $\widehat F(\widehat\lambda)=1$, the same derivative can also
be written as
\[
    \widehat F_\alpha(\widehat\lambda)
    =
    \frac{B-C}{B+\tau C}.
\]
Its denominator is positive, so $B>C>0$.

For $\alpha=i\leq n$, if $\lambda_i\geq0$, then
\[
    \frac{B}{B+\lambda_iC}
    =
    \frac{1}{1+\lambda_i\frac{C}{B}}
    \geq
    \frac{1}{1+\lambda_i}.
\]
If $\lambda_i\leq0$, then
\[
    0<B+\lambda_iC\leq B,
    \qquad
    \frac{B}{B+\lambda_iC}\geq1.
\]
These estimates prove both directional bounds in \eqref{3.14},
with $c(n,k)=c_0$.

Finally, summing the derivatives in all $n+1$ variables gives
\[
\begin{aligned}
    \sum_{\alpha=1}^{n+1}
    \widehat F_\alpha(\widehat\lambda)
    &=
    (n-k+2)
    -(n-k+3)
    \frac{
        \sigma_k(\widehat\lambda)
        \sigma_{k-2}(\widehat\lambda)
    }{
        \sigma_{k-1}(\widehat\lambda)^2
    }
    \\
    &\leq n-k+2.
\end{aligned}
\]
Every summand is positive, so the same upper bound holds for
$\sum_{i=1}^n a_i$.
\end{proof}

Apply the lemma to
\[
    \lambda=\frac{\kappa[v]}{s}.
\]
Since $F$ is homogeneous of degree one and $F(\lambda)=1$,
\[
    F_i(\kappa[v])
    =
    F_i(\lambda)
    =
    \frac{
        \sigma_{k-1}(\lambda|i)-\sigma_{k-3}(\lambda|i)
    }{
        2\sigma_{k-2}(\lambda)
    }.
\]
Define the tensor $a^{ij}$ on $\Sigma_v$ by
\[
    a^{ij}=\chi F^{ij},
    \qquad
    \chi=
    \frac{2\sigma_{k-2}(\lambda)}
         {\sigma_{k-1}(\lambda)+\sigma_{k-2}(\lambda)}>0,
\]
where $F^{ij}$ is evaluated on the upper graph. In a principal
orthonormal frame, $a^{ij}=a_i\delta_{ij}$, with $a_i$ as in
\eqref{3.13}. Homogeneity and differentiation of the constant
curvature equation give
\begin{equation}\label{3.15}
\begin{aligned}
    m:=\sum_{i=1}^n a_i\kappa_i
    &=
    \chi s
    =
    \frac{2s\sigma_{k-2}(\lambda)}
         {\sigma_{k-1}(\lambda)+\sigma_{k-2}(\lambda)},
    \qquad
    0<m\leq2s,
    \\
    \sum_{i=1}^n a_i h_{iij}
    &=
    \chi\sum_{i=1}^nF_i h_{iij}
    =0,
    \qquad
    1\leq j\leq n.
\end{aligned}
\end{equation}
Here $h$ denotes the second fundamental form of $\Sigma_v$.
Moreover,
\[
    (1+|\lambda_i|)^2
    =
    \frac{(s+|\kappa_i|)^2}{s^2}
    \leq
    \frac{1+s^2}{s^2}(1+\kappa_i^2),
\]
so Lemma~\ref{normalized-linearization} gives
\[
    a_i(1+\kappa_i^2)
    \geq
    \frac{c_0s^2}{1+s^2}.
\]
Since $s=s_0(n,k)$ is fixed, we obtain
\begin{equation}\label{3.16}
\begin{gathered}
    a^{ij}g_{ij}\leq C,
    \qquad
    0<a^{ij}h_{ij}=m\leq C,
    \\
    a_i(1+\kappa_i^2)\geq c,
    \qquad
    \kappa_i\leq0
    \quad\Longrightarrow\quad
    a_i\geq c.
\end{gathered}
\end{equation}
We write
\[
    \Delta_a\phi=a^{ij}\nabla^2_{ij}\phi,
    \qquad
    |\nabla\phi|_a^2=a^{ij}\phi_i\phi_j,
    \qquad
    \mathcal E_a=\sum_{i=1}^n a_i\kappa_i^2.
\]

The factor $\chi$ need not have a uniform positive lower bound.
This is why the initial separation estimate used the unnormalized
tensor $T_v$. For propagation, only the gradient term is needed.
Indeed, multiplying the upper-test inequality \eqref{3.9} by $\chi$
gives
\[
    \Delta_a\phi
    \geq
    \gamma|\nabla\phi|_a^2
    +\frac{\chi\delta}{d}
    \geq
    \gamma|\nabla\phi|_a^2.
\]
Thus, in the same upper-support sense at points with an interior
nearest lower point,
\begin{equation}\label{3.17}
    \Delta_a b
    \geq
    \gamma|\nabla b|_a^2,
    \qquad
    b=-\log d.
\end{equation}

\subsubsection*{The geometric auxiliary function.}

Let $y_0,d_*,r_*$ be given by the initial separation estimate, so that
\[
    y_0\in\overline{B}_{\frac{R_0}{16}},
    \qquad
    d\geq d_*
    \quad\text{in }B_{r_*}(y_0),
    \qquad
    0<r_*
    \leq
    \min\left\{
        \frac{R_0}{64},1
    \right\}.
\]
Both $d_*$ and $r_*$ have positive lower bounds depending only on
$n,k,R,K$. Retain the bound $|Dv|\leq K$ on
$B_{\frac{7R_0}{8}}$, and set
\[
    \ell=\frac{3R_0}{4},
    \qquad
    \mathcal B=B_\ell(y_0)
    \subset
    B_{\frac{13R_0}{16}}
    \Subset
    B_{\frac{7R_0}{8}}.
\]
All constants below depend only on $n,k,R,K$.

Following \cite{QiuYan}, define
\[
    z=y-y_0,
    \qquad
    r=|z|,
    \qquad
    W=W_v=\frac{1}{\langle N_v,E_{n+1}\rangle},
    \qquad
    W_0=\sqrt{1+K^2}.
\]
At a fixed point with $r>0$, choose a principal orthonormal frame
$\{e_i\}_{i=1}^n$ on $\Sigma_v$, and set
\[
    R_i=\langle(z,0),e_i\rangle,
    \qquad
    t_i=\langle e_i,E_{n+1}\rangle,
    \qquad
    \omega=\langle(z,0),N_v\rangle,
\]
and
\[
    Z_i
    =
    W\bigl[R_i+Wt_i(r-\omega)\bigr].
\]
Define
\[
    \Phi
    =
    -W\langle X_v(y)-X_v(y_0),N_v\rangle
    +r(W-1).
\]
Since $E_{n+1}=(0,\ldots,0,-1)$ and
$W\omega=z\cdot Dv$, we may also write
\[
\begin{aligned}
    \Phi
    &=
    v(y)-v(y_0)+W(r-\omega)-r
    \\
    &=
    v(y)-v(y_0)-z\cdot Dv+r(W-1).
\end{aligned}
\]
Thus $\Phi(y_0)=0$ and $|\Phi|\leq C(K)r$ on $\mathcal B$.
The function is continuous at $y_0$ and smooth away from $y_0$;
all differential calculations below are made at points with $r>0$.

The graph identities give
\[
\begin{gathered}
    v_i=-t_i,
    \qquad
    \nabla_jt_i=-\frac{h_{ij}}{W},
    \qquad
    \nabla_jR_i
    =
    \delta_{ij}-t_it_j-h_{ij}\omega,
    \\
    r_i=\frac{R_i}{r},
    \qquad
    W_i=-W^2\kappa_it_i,
    \qquad
    \omega_i=\kappa_iR_i-\frac{t_i}{W}.
\end{gathered}
\]
Differentiating $\Phi$ gives
\begin{equation}\label{3.18}
\begin{aligned}
    \Phi_i
    &=
    -t_i+(r-\omega)W_i+(W-1)r_i-W\omega_i
    \\
    &=
    \frac{W-1}{r}R_i
    -\kappa_i\bigl[WR_i+W^2t_i(r-\omega)\bigr]
    \\
    &=
    \frac{W-1}{r}R_i-\kappa_iZ_i.
\end{aligned}
\end{equation}

To compute its weighted Laplacian, the Codazzi equations and
\eqref{3.15} give
\[
\begin{aligned}
    \Delta_a\!\left(\frac1W\right)
    &=
    \sum_{i,j=1}^n a_i h_{iij}t_j
    -\frac{\mathcal E_a}{W}
    =
    -\frac{\mathcal E_a}{W},
    \\
    \Delta_aW
    &=
    W\mathcal E_a
    +
    2W^3\sum_i a_i\kappa_i^2t_i^2,
    \qquad
    \Delta_av=\frac{m}{W}.
\end{aligned}
\]
Similarly,
\[
\begin{aligned}
    \Delta_a\omega
    &=
    \sum_{i,j=1}^n a_i h_{iij}R_j
    +
    \sum_i a_i\kappa_i
    \left(
        1+\frac1{W^2}-2t_i^2
    \right)
    -\omega\mathcal E_a
    \\
    &=
    \left(
        1+\frac1{W^2}
    \right)m
    -
    2\sum_i a_i\kappa_it_i^2
    -\omega\mathcal E_a,
    \\
    \Delta_ar
    &=
    \frac1r
    \left[
        \sum_i a_i
        \left(
            1-t_i^2-\frac{R_i^2}{r^2}
        \right)
        -m\omega
    \right].
\end{aligned}
\]
The product rule yields
\[
\begin{aligned}
    \Delta_a\Phi
    &=
    \frac{m}{W}
    +(r-\omega)\Delta_aW
    +(W-1)\Delta_ar
    -W\Delta_a\omega
    +2\sum_i a_iW_i(r_i-\omega_i).
\end{aligned}
\]
Substituting the preceding identities and canceling the terms linear
in $\sum_i a_i\kappa_it_i^2$, we obtain
\begin{equation}\label{3.19}
\begin{aligned}
    \Delta_a\Phi
    ={}&
    \sum_i a_i\kappa_i^2
    \left\{
        rW(1+2W^2t_i^2)
        +2W^2t_i(R_i-Wt_i\omega)
    \right\}
    \\
    &-
    \frac{2W^2}{r}
    \sum_i a_i\kappa_it_iR_i
    -mW
    \\
    &+
    \frac{W-1}{r}
    \left[
        \sum_i a_i
        \left(
            1-t_i^2-\frac{R_i^2}{r^2}
        \right)
        -m\omega
    \right].
\end{aligned}
\end{equation}
Only the contracted derivative identity in \eqref{3.15} has been
used; no derivative of $a^{ij}$ occurs.

We next estimate the coefficient in braces. Since
\[
    \omega=\frac{z\cdot Dv}{W},
    \qquad
    \sum_iR_i^2=r^2-\omega^2,
\]
we have $r-\omega>0$ and
$R_i^2\leq r^2-\omega^2$. Completing the square gives
\[
\begin{aligned}
    &rW(1+2W^2t_i^2)
    +2W^2t_i(R_i-Wt_i\omega)
    \\
    &\quad=
    2W^3(r-\omega)
    \left(
        t_i+\frac{R_i}{2W(r-\omega)}
    \right)^2
    +
    W\left(
        r-\frac{R_i^2}{2(r-\omega)}
    \right)
    \\
    &\quad\geq
    W\left(
        r-\frac{r^2-\omega^2}{2(r-\omega)}
    \right)
    =
    \frac{W(r-\omega)}2.
\end{aligned}
\]
Moreover,
\[
    W(r-\omega)
    \geq
    r(W-|Dv|)
    =
    \frac{r}{W+|Dv|}
    \geq
    \frac{r}{2W}.
\]
Thus the coefficient is at least $\frac{r}{4W}$. Young's inequality gives
\[
\begin{aligned}
    \frac{2W^2}{r}
    \sum_i a_i|\kappa_it_iR_i|
    &\leq
    \frac{r}{8W}\mathcal E_a
    +
    \frac{8W^5}{r^3}
    \sum_i a_it_i^2R_i^2
    \\
    &\leq
    \frac{r}{8W}\mathcal E_a
    +
    \frac{C(K)}{r}\sum_i a_i.
\end{aligned}
\]
The ambient vectors $E_{n+1}$ and $\frac{(z,0)}{r}$ are orthonormal.
Consequently,
\[
    t_i^2+\frac{R_i^2}{r^2}\leq1.
\]
Using this fact, \eqref{3.16}, and $|\omega|\leq r\leq\ell$,
we deduce from \eqref{3.19} that
\[
    \Delta_a\Phi
    \geq
    \frac{r}{8W}\mathcal E_a
    -\frac{C}{r}
    -Cm.
\]
Since $m\leq C$ and $r\leq\ell$, the last constant term can be
absorbed into $\frac{C}{r}$. Therefore
\begin{equation}\label{3.20}
    \Delta_a\Phi
    \geq
    cr\mathcal E_a-\frac{C}{r}.
\end{equation}

For clarity, write
$\mathbf R=(R_1,\ldots,R_n)$,
$\mathbf Z=(Z_1,\ldots,Z_n)$, and
$\mathbf t=(t_1,\ldots,t_n)$. Orthogonal decomposition in the
ambient space gives
\[
    |\mathbf R|^2=r^2-\omega^2,
    \qquad
    |\mathbf t|^2=1-\frac1{W^2},
    \qquad
    \mathbf R\cdot\mathbf t=-\frac{\omega}{W}.
\]
It follows that
\[
    |\mathbf Z|
    \leq
    W|\mathbf R|
    +W^2(r-\omega)|\mathbf t|
    \leq
    C(K)r,
\]
and
\[
\begin{aligned}
    \mathbf R\cdot\mathbf Z
    &=
    W(r^2-\omega^2)
    +W^2(r-\omega)
    \left(
        -\frac{\omega}{W}
    \right)
    \\
    &=
    Wr(r-\omega)
    =
    r^2W-rz\cdot Dv
    \\
    &\geq
    \frac{r^2}{W+|Dv|}
    \geq
    \frac{r^2}{\sqrt{1+K^2}+K}.
\end{aligned}
\]
Thus, with
$c_Z=(\sqrt{1+K^2}+K)^{-1}$,
\begin{equation}\label{3.21}
    |\mathbf R|\leq r,
    \qquad
    |\mathbf Z|\leq C(K)r,
    \qquad
    \mathbf R\cdot\mathbf Z\geq c_Zr^2.
\end{equation}

\subsubsection*{Propagation of the lower bound.}

On $\mathcal B=B_\ell(y_0)$, put
\[
    \rho=\ell^2-r^2
    =
    \left(
        \frac{3R_0}{4}
    \right)^2-r^2.
\]
Initially choose $\beta$ and $S$ so that
\[
    0<\beta
    \leq
    r_*
    \min\left\{
        1,
        \frac{2}{W_0\ell^2}
    \right\},
\]
and
\[
    S>
    \max\left\{
        1,
        -\log d_*,
        \frac{2}{\gamma},
        \log\left(
            \frac{16}{3R_0}
        \right)
    \right\}.
\]
Their final values will be fixed below using only $n,k,R,K$.
Consider
\[
    \mathcal Q(y)
    =
    \rho(y)^2e^{\beta\Phi(y)}
    \max\{b(y),S\}.
\]
This function is continuous on $\overline{\mathcal B}$, positive
in $\mathcal B$, and zero on its boundary. Let $\bar y$ be an
interior maximum point. Equivalently, $\bar y$ maximizes
\[
    2\log\rho+\beta\Phi+\log\max\{b,S\}.
\]

If $b(\bar y)\leq S$, then
\[
    \max_{\overline{\mathcal B}}\mathcal Q
    \leq
    \ell^4
    e^{\beta\|\Phi\|_{L^\infty(\mathcal B)}}S.
\]
This already gives the required bound on any region where $\rho$
has a fixed positive lower bound.

We may therefore assume $b(\bar y)>S$. Since
$b\leq-\log d_*$ on $B_{r_*}(y_0)$,
\[
    r=|\bar y-y_0|\geq r_*>0.
\]
All differential identities for $\Phi$ are therefore valid near
$\bar y$. Moreover,
\[
    d(\bar y)
    =
    e^{-b(\bar y)}
    <
    e^{-S}
    <
    \frac{3R_0}{16},
\]
whereas
\[
    \operatorname{dist}
    (\bar y,\partial B_{R_0})
    \geq
    R_0-|\bar y|
    >
    \frac{3R_0}{16}.
\]
Thus every nearest lower point associated with $\bar y$ is interior,
and \eqref{3.17} is applicable there.

The maximality of $\mathcal Q$ gives, near $\bar y$,
\[
    b(y)
    \leq
    \max\{b(y),S\}
    \leq
    b(\bar y)
    \frac{\rho(\bar y)^2}{\rho(y)^2}
    e^{\beta(\Phi(\bar y)-\Phi(y))}.
\]
The right-hand side defines a smooth positive upper test
\[
    \widetilde b(y)
    =
    b(\bar y)
    \frac{\rho(\bar y)^2}{\rho(y)^2}
    e^{\beta(\Phi(\bar y)-\Phi(y))}
\]
for $b$ at $\bar y$. Henceforth, derivatives of $b$ at $\bar y$
mean derivatives of this upper test. By construction,
\[
    2\log\rho+\beta\Phi+\log\widetilde b
\]
is constant near $\bar y$.

All quantities below are evaluated at $\bar y$ in a principal
orthonormal frame. Since $\rho_i=-2R_i$, differentiation and
\eqref{3.18} give
\[
    0
    =
    -\frac{4R_i}{\rho}
    +\beta\Phi_i
    +\frac{b_i}{b},
\]
and hence
\[
    \frac{b_i}{b}
    =
    \mathcal A R_i+\beta\kappa_iZ_i,
    \qquad
    \mathcal A
    =
    \frac4\rho
    -\frac{\beta(W-1)}r.
\]
The preliminary restriction on $\beta$ gives
\[
    0
    \leq
    \frac{\beta(W-1)}r
    \leq
    \frac2{\ell^2}
    \leq
    \frac2\rho,
\]
and therefore
\[
    \frac2\rho
    \leq
    \mathcal A
    \leq
    \frac4\rho.
\]
Set
\[
    \mathcal S
    =
    \sum_i
    a_i
    \bigl(
        \mathcal A R_i+\beta\kappa_iZ_i
    \bigr)^2
    =
    \frac{|\nabla b|_a^2}{b^2}.
\]

For the second derivatives of the cutoff,
\[
    \rho_{ij}
    =
    -2(
        \delta_{ij}
        -t_it_j
        -h_{ij}\omega
    ),
\]
and hence
\[
\begin{aligned}
    \Delta_a(2\log\rho)
    &=
    -\frac4\rho
    \sum_i a_i(1-t_i^2)
    +
    \frac{4m\omega}{\rho}
    -
    \frac8{\rho^2}
    \sum_i a_iR_i^2
    \geq
    -\frac{C}{\rho^2}.
\end{aligned}
\]
Here we used \eqref{3.16}, $|\omega|\leq\ell$, and
$0<\rho\leq\ell^2$.

Also, \eqref{3.17} and $b>S>\frac{2}{\gamma}$ give
\[
\begin{aligned}
    \Delta_a\log b
    &=
    \frac{\Delta_ab}{b}
    -
    \frac{|\nabla b|_a^2}{b^2}
    \\
    &\geq
    (\gamma b-1)\mathcal S
    \geq
    \frac{\gamma b}{2}\mathcal S.
\end{aligned}
\]
Combining these estimates with \eqref{3.20}, we obtain
\[
\begin{aligned}
    0
    &=
    \Delta_a
    \left(
        2\log\rho+\beta\Phi+\log b
    \right)
    \\
    &\geq
    -\frac{C}{\rho^2}
    +
    c\beta r\mathcal E_a
    -
    \frac{C\beta}{r}
    +
    \frac{\gamma b}{2}\mathcal S.
\end{aligned}
\]
Since $\beta\leq r_*\leq r$, we have $\frac{\beta}{r}\leq1$.
As $\rho\leq\ell^2$, the resulting constant term can be absorbed
into $\frac{C}{\rho^2}$. Thus
\begin{equation}\label{3.22}
    0
    \geq
    -\frac{C}{\rho^2}
    +
    c\beta r\mathcal E_a
    +
    \frac{\gamma b}{2}\mathcal S.
\end{equation}
In particular,
\[
    \mathcal E_a
    \leq
    \frac{C}{\beta r\rho^2},
    \qquad
    \mathcal S
    \leq
    \frac{C}{b\rho^2}.
\]
The constants in these estimates are independent of the choices
of $\beta$ and $S$, subject to the preliminary restrictions.

Using $\mathcal A\geq\frac{2}{\rho}$ and \eqref{3.21}, we obtain
\begin{equation}\label{3.23}
\begin{aligned}
    \sum_i a_iR_i^2
    &\leq
    \frac{2}{\mathcal A^2}
    \left(
        \mathcal S
        +
        \beta^2
        \sum_i a_i\kappa_i^2Z_i^2
    \right)
    \\
    &\leq
    C\rho^2
    \left(
        \mathcal S
        +
        \beta^2r^2\mathcal E_a
    \right)
    \\
    &\leq
    C\left(
        \frac1b+\beta r
    \right)
    \leq
    C\left(
        \frac1b+\beta
    \right).
\end{aligned}
\end{equation}

Define
\[
    I_-=\{i:\kappa_i\leq1\},
    \qquad
    I_+=\{i:\kappa_i>1\}.
\]
For $i\in I_-$, the directional estimates in \eqref{3.16} give
$a_i\geq c$: the sign bound applies when $\kappa_i\leq0$, while
$1+\kappa_i^2\leq2$ when $0<\kappa_i\leq1$. For $i\in I_+$,
\[
    a_i\kappa_i^2
    \geq
    c\frac{\kappa_i^2}{1+\kappa_i^2}
    \geq
    \frac c2.
\]
It follows from \eqref{3.23} that
\[
    \sum_{i\in I_-}R_i^2
    \leq
    C\sum_i a_iR_i^2
    \leq
    C\left(
        \frac1b+\beta
    \right).
\]
Hence, by Cauchy--Schwarz and \eqref{3.21},
\[
\begin{aligned}
    \left|
        \sum_{i\in I_-}R_iZ_i
    \right|
    &\leq
    \left(
        \sum_{i\in I_-}R_i^2
    \right)^{\frac12}
    |\mathbf Z|
    \\
    &\leq
    C_2r
    \left(
        \frac1S+\beta
    \right)^{\frac12}.
\end{aligned}
\]

We now fix
\[
    \beta=c_1r_*^2,
    \qquad
    S\geq C_1r_*^{-2},
\]
where $c_1>0$ is sufficiently small and $C_1>0$ sufficiently large
that
\[
    c_1
    \leq
    \min\left\{
        1,
        \frac2{W_0\ell^2}
    \right\},
    \qquad
    C_2
    \left(
        c_1+\frac1{C_1}
    \right)^{\frac12}
    \leq
    \frac{c_Z}{2}.
\]
Take $S$ also to satisfy its preliminary lower bounds. These choices
depend only on $n,k,R,K$ and satisfy all earlier restrictions because
$r_*\leq1$.

Since $r\geq r_*$,
\[
    \left|
        \sum_{i\in I_-}R_iZ_i
    \right|
    \leq
    \frac{c_Z}{2}rr_*
    \leq
    \frac{c_Z}{2}r^2.
\]
Together with \eqref{3.21}, this gives
\[
    \sum_{i\in I_+}R_iZ_i
    =
    \mathbf R\cdot\mathbf Z
    -
    \sum_{i\in I_-}R_iZ_i
    \geq
    \frac{c_Z}{2}r^2.
\]
Hence there exists $j\in I_+$ such that
\[
    R_jZ_j
    \geq
    \frac{c_Z}{2n}r^2>0,
    \qquad
    |Z_j|
    \geq
    \frac{c_Z}{2n}r,
\]
where we used $|R_j|\leq r$. Since
$\mathcal A>0$, $\beta>0$, and $\kappa_j>1$,
the terms $\mathcal A R_j$ and $\beta\kappa_jZ_j$ have the same
sign. Consequently,
\[
\begin{aligned}
    \mathcal S
    &\geq
    a_j
    \bigl(
        \mathcal A R_j+\beta\kappa_jZ_j
    \bigr)^2
    \\
    &\geq
    a_j\beta^2\kappa_j^2Z_j^2
    \geq
    c\beta^2r^2.
\end{aligned}
\]
Returning to \eqref{3.22}, we conclude that
\[
    b\rho^2
    \leq
    \frac{C}{\beta^2r^2}
    \leq
    \frac{C}{\beta^2r_*^2}
    \leq
    Cr_*^{-6}.
\]

Thus, whether $b(\bar y)\leq S$ or $b(\bar y)>S$, maximality gives
\[
\begin{aligned}
    \rho(y)^2
    e^{\beta\Phi(y)}
    \max\{b(y),S\}
    &\leq
    \mathcal Q(\bar y)
    \\
    &\leq
    e^{\beta\|\Phi\|_{L^\infty(\mathcal B)}}
    \max\left\{
        \ell^4S,
        Cr_*^{-6}
    \right\}.
\end{aligned}
\]
Here $S$ may be chosen as the maximum of the stated lower bounds
plus $1$, so $S$ depends only on $n,k,R,K$.

For $y\in B_{\frac{R}{2}}=B_{\frac{2R_0}{3}}$,
\[
    |y-y_0|
    \leq
    \frac{2R_0}{3}
    +
    \frac{R_0}{16}
    =
    \frac{35R_0}{48}
    <
    \ell,
\]
and hence
\[
    \rho(y)
    \geq
    \left(
        \frac{3R_0}{4}
    \right)^2
    -
    \left(
        \frac{35R_0}{48}
    \right)^2
    =
    \frac{71R_0^2}{2304}
    >0.
\]
Since $|\Phi|\leq C$ on $\mathcal B$ and $r_*$ has a positive lower
bound depending only on $n,k,R,K$, the preceding maximum estimate
gives $b\leq C$ on $B_{\frac{R}{2}}$. Using \eqref{3.2}, we obtain
\begin{equation}\label{3.24}
    v-u\geq d=e^{-b}\geq c>0
    \qquad
    \text{in }B_{\frac{R}{2}}.
\end{equation}
This proves Theorem~\ref{thm:separation}.

\section{The two-surface Pogorelov estimate}\label{pog-section}

We prove a weighted curvature estimate for two admissible graphs.
The upper graph supplies the comparison term in the paired
second-derivative calculation; no equation is prescribed for it.

\begin{proposition}\label{pog}
Let $F$ be positive, symmetric, smooth, elliptic, concave, and
homogeneous of degree one on an open convex symmetric cone
$\Gamma\subset\Gamma_1$. Let $\Omega\subset\mathbb{R}^n$ be a
bounded domain, and suppose that
$u,v\in C^4(\Omega)\cap C^2(\overline{\Omega})$ satisfy
\[
    \kappa[u],\kappa[v]\in\Gamma,
    \qquad
    F(\kappa[u])=1
    \quad\text{in }\Omega,
\]
and
\[
    \rho=v-u>0\quad\text{in }\Omega,
    \qquad
    \rho=0\quad\text{on }\partial\Omega,
    \qquad
    \|u\|_{C^1(\Omega)}+\|v\|_{C^1(\Omega)}\leq K.
\]
Then there exist $\beta=\beta(n,K)>0$ and $C=C(n,K)>0$ such that
\[
    \sup_{\Omega}\rho^\beta\kappa_1[u]\leq C.
\]
\end{proposition}

\begin{proof}
We may assume $K\geq1$. Unless otherwise indicated, curvatures,
covariant derivatives, and the linearized operator in this proof
are taken on the lower graph $\Sigma_u$. In a principal orthonormal
frame, write $F_i=F^{ii}$ and
\[
    \mathcal E=\sum_iF_i\kappa_i^2.
\]
Set
\[
    w=W_u^{-1},
    \qquad
    a_*=\frac{1}{2\sqrt{1+K^2}},
    \qquad
    G=\log\kappa_1-\log(w-a_*).
\]
Since $\Gamma\subset\Gamma_1$, we have $\kappa_1>0$ in $\Omega$.
Also
\[
    2a_*\leq w\leq1,
\]
so $w-a_*\geq a_*>0$.

We first record a differential inequality for a smooth upper test
$b$ of $\log\kappa_1$. At the contact point, let $m$ be the
multiplicity of $\kappa_1$, and set $\psi=e^b$. The standard
upper-support formula for the largest eigenvalue gives
\[
\begin{gathered}
    h_{\alpha\beta i}
    =
    \psi_i\delta_{\alpha\beta},
    \qquad
    1\leq\alpha,\beta\leq m,
    \\
    \psi_{ii}
    \geq
    h_{11ii}
    +
    2\sum_{j>m}
    \frac{h_{1ji}^2}{\kappa_1-\kappa_j},
    \qquad
    1\leq i\leq n.
\end{gathered}
\]
In particular,
\[
    \psi_i=h_{11i},
\]
and Codazzi implies
\[
    \psi_i=h_{1i1}=0,
    \qquad
    2\leq i\leq m.
\]

Differentiating $F=1$ twice and commuting covariant derivatives
gives
\[
    \sum_iF_i h_{11ii}
    =
    -F^{ij,rs}h_{ij1}h_{rs1}
    +
    \kappa_1^2-\kappa_1\mathcal E.
\]
The spectral second-derivative formula, together with concavity, gives
\[
    -F^{ij,rs}h_{ij1}h_{rs1}
    \geq
    2\sum_{i>m}
    \frac{F_i-F_1}{\kappa_1-\kappa_i}
    h_{1i1}^2;
\]
see \cite[Theorem~5.1]{AndrewsPinching}. All omitted terms are
nonnegative. Combining this estimate with the terms corresponding
to the first principal direction $e_1$ in the second variation of
$\psi$, and using $h_{1i1}=h_{11i}$, we obtain
\[
\begin{aligned}
    &\frac{2}{\kappa_1}
    \left(
        \frac{F_i-F_1}{\kappa_1-\kappa_i}
        +
        \frac{F_1}{\kappa_1-\kappa_i}
    \right)
    h_{11i}^2
    -
    \frac{F_i h_{11i}^2}{\kappa_1^2}=
    \frac{\kappa_1+\kappa_i}
         {\kappa_1-\kappa_i}
    F_i b_i^2,
    \qquad
    i>m.
\end{aligned}
\]
Consequently,
\begin{equation}\label{4.1}
    \Delta_F b
    \geq
    -F_1b_1^2
    +
    \sum_{i>m}
    \frac{\kappa_1+\kappa_i}
         {\kappa_1-\kappa_i}
    F_i b_i^2
    +
    \kappa_1-\mathcal E,
    \qquad
    b_i=0
    \quad(2\leq i\leq m).
\end{equation}

For later use, define
\[
    \gamma_1=-1,
    \qquad
    \gamma_i=0
    \quad(2\leq i\leq m),
    \qquad
    \gamma_i=
    \frac{\kappa_1+\kappa_i}
         {\kappa_1-\kappa_i}
    \quad(i>m).
\]
For $i>m$, positivity of the mean curvature implies
\[
    \kappa_i>-(n-1)\kappa_1,
\]
and hence
\[
    \gamma_i+1
    =
    \frac{2\kappa_1}
         {\kappa_1-\kappa_i}
    \geq
    \frac{2}{n}.
\]
Also, if $i>m$ and $\kappa_i\geq-\frac{\kappa_1}{2}$, then
\[
    \gamma_i\geq\frac{1}{3}.
\]

Choose $p\geq1$, to be fixed below, and set \( \beta=2p\). Let $x_0\in\Omega$ maximize $\rho^{2p}e^G$, and put
\[
    r=\rho(x_0)>0.
\]
Such an interior maximum exists because $\rho=0$ on
$\partial\Omega$ and $\kappa_1$ is bounded on
$\overline{\Omega}$.

For
\[
    \frac{r}{2}<\rho(x)<2r,
    \qquad
    |y-x|<\frac{r}{2K},
\]
define
\[
    \eta(x,y)
    =
    v(y)-u(x)
    -
    \frac{17r}{32}
    -
    \frac{8K^2}{r}|y-x|^2.
\]
Since
\[
    |D\rho|\leq K
    \qquad\text{and}\qquad
    \rho=0\quad\text{on }\partial\Omega,
\]
we have
\[
    \operatorname{dist}(x,\partial\Omega)
    \geq
    \frac{\rho(x)}{K}
    >
    \frac{r}{2K}.
\]
Thus the $y$-ball lies in $\Omega$. The gradient bound and
\[
    K|y-x|
    -
    \frac{8K^2}{r}|y-x|^2
    \leq
    \frac{r}{32}
\]
give
\[
    \eta(x,y)
    \leq
    \rho(x)-\frac{r}{2},
    \qquad
    \eta(x_0,x_0)=\frac{15r}{32}.
\]

Maximize
\[
    e^{G(x)}\eta(x,y)^p
\]
on the region where $\eta>0$. The maximum cannot occur where
$\eta=0$ or $\rho=\frac{r}{2}$. On
\[
    |y-x|=\frac{r}{2K},
\]
we have
\[
    \eta
    \leq
    \rho-\frac{65r}{32}
    \leq
    -\frac{r}{32}.
\]
On $\rho=2r$, maximality of $\rho^{2p}e^G$ gives
\[
    \frac{
        e^{G(x)}\eta(x,y)^p
    }{
        e^{G(x_0)}
        \left(\frac{15r}{32}\right)^p
    }
    \leq
    \frac{1}{2^{2p}}
    \left(
        \frac{16}{5}
    \right)^p
    =
    \left(
        \frac{4}{5}
    \right)^p
    <1.
\]
Hence the maximum is attained at an interior pair $(x_*,y_*)$.
Differentiation in $y$ gives
\begin{equation}\label{4.2}
    Dv(y_*)
    =
    \frac{16K^2}{r}
    (y_*-x_*),
    \qquad
    0<\eta_*:=\eta(x_*,y_*)
    \leq
    \frac{3r}{2}.
\end{equation}

We regard $\eta$ as a function on
$\Sigma_u\times\Sigma_v$ through the graph parametrizations.
All quantities below are evaluated at
\[
    X=X_u(x_*),
    \qquad
    Y=X_v(y_*).
\]
Let
\[
    c=\langle N_u,N_v\rangle.
\]
Since the graph normals are downward,
\[
    1+c
    =
    \frac{|N_u+N_v|^2}{2}
    \geq
    \frac{2}{1+K^2}
    >0.
\]

The tangent spaces are identified by the isometry
\[
    \mathcal RZ
    =
    Z
    -
    \frac{\langle N_v,Z\rangle}{1+c}
    (N_u+N_v),
    \qquad
    Z\in T_X\Sigma_u.
\]
Choose a principal orthonormal frame $\{e_i\}$ at $X$, and set
\[
    \widetilde e_i=\mathcal Re_i.
\]
Let $\pi$ denote horizontal projection. For the paired directions,
write
\[
    \eta_i
    =
    d\eta(e_i,\widetilde e_i),
\]
and
\[
    \mathcal H_i
    =
    \nabla^2_{\Sigma_u\times\Sigma_v}\eta
    \bigl(
        (e_i,\widetilde e_i),
        (e_i,\widetilde e_i)
    \bigr).
\]
Using \eqref{4.2}, differentiation gives
\[
    \eta_i
    =
    W_v\langle N_v,e_i\rangle,
    \qquad
    |\eta_i|\leq C(K),
\]
and
\[
    |\pi(\widetilde e_i-e_i)|^2
    \leq
    \frac{
        2\langle N_v,e_i\rangle^2
    }{
        1+c
    }
    \leq
    C(K)\eta_i^2.
\]
Moreover,
\[
    \mathcal H_i
    =
    W_v
    \bigl(
        h_v(\widetilde e_i,\widetilde e_i)
        -
        c\kappa_i
    \bigr)
    -
    \frac{16K^2}{r}
    |\pi(\widetilde e_i-e_i)|^2.
\]
Indeed, at the critical pair the ambient gradients of $\eta$
with respect to $X$ and $Y$ are
\[
    W_v N_v
    \qquad\text{and}\qquad
    -W_v N_v,
\]
respectively; the ambient Hessian contributes only the quadratic
penalty.

Let $A_v^{\mathcal R}$ denote the matrix with entries
\[
    h_v(\widetilde e_i,\widetilde e_j).
\]
It is admissible and has the same eigenvalues as $A_v$.
Concavity and homogeneity give
\[
\begin{aligned}
    0
    <
    F(A_v^{\mathcal R})
    &\leq
    F(A_u)
    +
    F^{ij}
    \bigl(
        (A_v^{\mathcal R})_{ij}
        -
        (A_u)_{ij}
    \bigr)
    \\
    &=
    \sum_iF_i
    h_v(\widetilde e_i,\widetilde e_i).
\end{aligned}
\]
Since
\[
    \sum_iF_i\kappa_i=1,
\]
the paired trace satisfies
\begin{equation}\label{4.3}
    \sum_iF_i\mathcal H_i
    \geq
    -C(K)
    -
    \frac{C(K)}{r}
    \sum_iF_i\eta_i^2.
\end{equation}
This is the only step in which the admissibility of $v$ is used;
no lower bound for $F(\kappa[v])$ is required.

To handle a possible repeated largest curvature, choose a smooth
local map
\[
    \mathcal J:\Sigma_u\to\Sigma_v
\]
such that
\[
    \mathcal J(X)=Y,
    \qquad
    d\mathcal J_X=\mathcal R,
    \qquad
    (\nabla d\mathcal J)_X=0.
\]
Such a map is obtained by using the linear map $\mathcal R$ in
geodesic coordinates at $X$ and $Y$. The product maximum supplies
the smooth upper test
\[
    b(X')
    =
    \log\kappa_1(X)
    +
    \log
    \frac{
        w(X')-a_*
    }{
        w(X)-a_*
    }
    +
    p\log
    \frac{
        \eta_*
    }{
        \eta(X',\mathcal J(X'))
    }
\]
for $\log\kappa_1$ at $X$. In particular, \eqref{4.1} applies
without assuming that the largest eigenvalue is simple. The
vanishing covariant second derivative of $\mathcal J$ ensures that
the Hessian of the pulled-back cutoff is the paired Hessian used in
\eqref{4.3}.

Put
\[
    \zeta_i
    =
    \frac{\eta_i}{\eta},
    \qquad
    s_i
    =
    -\frac{w_i}{w-a_*}.
\]
The first derivative identity is
\[
    b_i+p\zeta_i+s_i=0.
\]
Using
\[
    \Delta_Fw=-w\mathcal E,
\]
we have
\[
    \Delta_F[-\log(w-a_*)]
    =
    \frac{w}{w-a_*}\mathcal E
    +
    \sum_iF_i s_i^2.
\]
Tracing the Hessian of the pulled-back maximum function gives
\[
\begin{aligned}
    0
    ={}&
    \Delta_Fb
    +
    \Delta_F[-\log(w-a_*)]
    +
    \frac{p}{\eta}
    \sum_iF_i\mathcal H_i
    -
    p\sum_iF_i\zeta_i^2.
\end{aligned}
\]
Combining \eqref{4.1} and \eqref{4.3}, and using
\[
    \eta\leq\frac{3r}{2},
\]
we obtain
\begin{equation}\label{4.4}
\begin{aligned}
    0\geq{}&
    \kappa_1
    +
    \frac{a_*}{w-a_*}\mathcal E
    -
    \frac{Cp}{\eta}
    \\
    &+
    \sum_iF_i
    \left\{
        \gamma_i(p\zeta_i+s_i)^2
        -
        \Lambda\zeta_i^2
        +
        s_i^2
    \right\},
    \qquad
    p\leq\Lambda\leq C_0(K)p.
\end{aligned}
\end{equation}
Here we may take
\[
    \Lambda
    =
    p
    \left(
        1+
        C(K)\frac{\eta}{r}
    \right).
\]
In particular, the term $-\mathcal E$ in \eqref{4.1} has combined
with
\(
    \frac{w}{w-a_*}\mathcal E
\)
to give the positive curvature term in \eqref{4.4}.

We estimate the directional quadratic expression
\begin{equation}\label{4.5}
    Q_i
    =
    \gamma_i(p\zeta_i+s_i)^2
    -
    \Lambda\zeta_i^2
    +
    s_i^2.
\end{equation}
For $i>m$, completing the square gives
\[
    Q_i
    =
    (\gamma_i+1)
    \left(
        s_i
        +
        \frac{p\gamma_i}{\gamma_i+1}
        \zeta_i
    \right)^2
    +
    \left(
        \frac{p^2\gamma_i}{\gamma_i+1}
        -
        \Lambda
    \right)
    \zeta_i^2.
\]
Choose
\[
    p
    \geq
    \max\{1,4C_0(K)\}.
\]
If $i>m$ and
\(
    \kappa_i\geq-\frac{\kappa_1}{2},
\)
then
\[
    \frac{\gamma_i}{\gamma_i+1}
    \geq
    \frac{1}{4},
\]
and hence
\(
    Q_i\geq0.
\)

For $i>m$ with
\(
    \kappa_i<-\frac{\kappa_1}{2},
\)
the bound
\(
    \gamma_i+1\geq\frac{2}{n}
\)
gives
\[
    Q_i
    \geq
    -C(n,K)p^2\zeta_i^2.
\]

For $2\leq i\leq m$, \eqref{4.1} gives $b_i=0$, and hence
\[
    s_i=-p\zeta_i.
\]
Thus
\[
    Q_i
    =
    (p^2-\Lambda)\zeta_i^2
    \geq0.
\]

For the first principal direction $e_1$,
\[
    Q_1
    =
    -(p^2+\Lambda)\zeta_1^2
    -
    2p\zeta_1s_1,
\]
and
\[
    |s_1|
    \leq
    \frac{\kappa_1}{a_*}.
\]
The latter estimate follows from
\[
    w_1=-\kappa_1u_1
\]
and $|u_1|\leq1$, where $u_1$ is the intrinsic derivative of the
height function.

Since
\[
    |\zeta_i|
    \leq
    \frac{C(K)}{\eta},
\]
the bad directions satisfy
\[
    \sum_{\substack{
        i>m\\
        \kappa_i<-\frac{\kappa_1}{2}
    }}
    F_iQ_i
    \geq
    -
    \frac{Cp^2}
         {\eta^2\kappa_1^2}
    \mathcal E,
\]
while the first direction gives
\[
    F_1Q_1
    \geq
    -
    \left(
        \frac{Cp^2}
             {\eta^2\kappa_1^2}
        +
        \frac{Cp}
             {\eta\kappa_1}
    \right)
    \mathcal E.
\]
Therefore, with
\[
    T=\eta\kappa_1,
    \qquad
    c_a
    =
    \frac{a_*}{1-a_*}
    >0,
\]
inequality \eqref{4.4} yields
\[
    0
    \geq
    \kappa_1
    +
    \left(
        c_a
        -
        \frac{Cp^2}{T^2}
        -
        \frac{Cp}{T}
    \right)
    \mathcal E
    -
    \frac{Cp}{\eta}.
\]
If
\(
    T\geq Lp,
\)
where $L=L(n,K)$ is sufficiently large, the coefficient of
$\mathcal E$ is at least
\(
    \frac{c_a}{2}.
\)
Dropping this nonnegative term gives
\[
    T\leq Cp.
\]
If $T<Lp$, the same estimate holds after increasing $C$. Hence
\[
    \eta_*\kappa_1(x_*)
    \leq
    Cp,
    \qquad
    \eta_*e^{G(x_*)}
    \leq
    C,
\]
where $p$ is now fixed in terms of $n$ and $K$.

Finally, maximality at $(x_*,y_*)$ and \eqref{4.2} imply
\[
    \left(
        \frac{15r}{32}
    \right)^p
    e^{G(x_0)}
    \leq
    \eta_*^p e^{G(x_*)}
    \leq
    C\eta_*^{p-1}
    \leq
    C
    \left(
        \frac{3r}{2}
    \right)^{p-1}.
\]
Thus
\[
    re^{G(x_0)}
    \leq
    C.
\]
Since $p\geq1$ and
\(
    r=\rho(x_0)\leq2K,
\)
we conclude that
\[
    \sup_{\Omega}
    \rho^{2p}\kappa_1
    \leq
    \sup_{\Omega}
    \rho^{2p}e^G
    =
    r^{2p}e^{G(x_0)}
    \leq
    Cr^{2p-1}
    \leq
    C(n,K).
\]
Taking
\(
    \beta=2p
\)
completes the proof.
\end{proof}

\section{The comparison problem and proof of Theorem~\ref{thm:main}}
\label{sec:main-proof}

The separation estimate was proved under the assumption that an
admissible comparison graph is available. We now construct this graph
on a sufficiently small ball and complete the proof of
Theorem~\ref{thm:main}. The estimates used only for the existence
argument may depend on higher norms of the fixed boundary data.
The separation and Pogorelov estimates do not use these higher norms.

\subsection{A tangential energy estimate}

The following algebraic estimate will be used to control the full
Hessian energy by its tangential part in the boundary estimates.

\begin{lemma}\label{energy}
Let $3\leq k\leq n$ and let $\kappa\in\Gamma_k$ satisfy
$F(\kappa)=s>0$. Then, for each $1\leq i\leq n$,
\[
    \sum_{j=1}^nF_j\kappa_j^2
    \leq
    \frac{n+2}{k}\sum_{j\neq i}F_j\kappa_j^2
    +
    \frac{n+2-k}{n-k+1}s^2
    \sum_{j=1}^nF_j.
\]
\end{lemma}

\begin{proof}
Set
\[
    \widehat\kappa=(\kappa,s,-s)\in\mathbb R^{n+2},
\]
with the convention $\sigma_j=0$ for $j<0$. Then
\[
    \sigma_j(\widehat\kappa)
    =
    \sigma_j(\kappa)-s^2\sigma_{j-2}(\kappa).
\]
Newton--Maclaurin gives, for $2\leq j<k$,
\[
    \frac{\sigma_j(\kappa)}
         {\sigma_{j-2}(\kappa)}
    \geq
    \frac{
        \binom{n}{j}\binom{n}{k-2}
    }{
        \binom{n}{j-2}\binom{n}{k}
    }s^2
    >
    s^2.
\]
Thus $\sigma_j(\widehat\kappa)>0$ for $1\leq j<k$, while
$\sigma_k(\widehat\kappa)=0$. Hence
\[
    \widehat\kappa\in\partial\Gamma_k.
\]

Write
\[
    T_\alpha
    =
    \sigma_{k-1}(\widehat\kappa|\alpha)
    \geq0,
    \qquad
    1\leq\alpha\leq n+2.
\]
We claim that
\[
    T_\alpha\widehat\kappa_\alpha^2
    \leq
    \frac{n+2-k}{n+2}
    \sum_{\beta=1}^{n+2}
    T_\beta\widehat\kappa_\beta^2.
\]
There is nothing to prove if $T_\alpha=0$. Otherwise, set
\[
    b=T_\alpha,
    \qquad
    a=\sigma_k(\widehat\kappa|\alpha),
    \qquad
    e=\sigma_{k+1}(\widehat\kappa|\alpha).
\]
Since
\[
    a+\widehat\kappa_\alpha b=0,
\]
the elementary symmetric identities give
\[
    T_\alpha\widehat\kappa_\alpha^2
    =
    \frac{a^2}{b},
\]
and
\[
    \sum_\beta T_\beta\widehat\kappa_\beta^2
    =
    -(k+1)\sigma_{k+1}(\widehat\kappa)
    =
    (k+1)
    \left(
        \frac{a^2}{b}-e
    \right).
\]
Newton's inequality for the remaining $n+1$ entries yields
\[
    (k+1)e
    \leq
    \frac{k(n+1-k)}{n+2-k}
    \frac{a^2}{b},
\]
which is immediate when $e\leq0$. Therefore
\[
    \sum_\beta T_\beta\widehat\kappa_\beta^2
    \geq
    \frac{n+2}{n+2-k}\frac{a^2}{b},
\]
and the claim follows.

For the original $n$ entries,
\[
    T_i
    =
    \sigma_{k-1}(\kappa|i)
    -
    s^2\sigma_{k-3}(\kappa|i)
    =
    2s\sigma_{k-2}(\kappa)F_i.
\]
The two additional entries contribute
\[
    s^2(T_{n+1}+T_{n+2})
    =
    2s^2\sigma_{k-1}(\kappa).
\]
Consequently, with
\[
    \mathcal E=\sum_jF_j\kappa_j^2,
\]
we obtain
\[
    F_i\kappa_i^2
    \leq
    \frac{n+2-k}{n+2}
    \left(
        \mathcal E
        +
        s\frac{\sigma_{k-1}}{\sigma_{k-2}}
    \right).
\]
Rearranging gives
\[
    \mathcal E
    \leq
    \frac{n+2}{k}
    \sum_{j\neq i}F_j\kappa_j^2
    +
    \frac{n+2-k}{k}
    s\frac{\sigma_{k-1}}{\sigma_{k-2}}.
\]

Finally, Newton's inequalities imply
\[
    s^2\sigma_{k-3}
    =
    \frac{\sigma_k\sigma_{k-3}}{\sigma_{k-2}}
    \leq
    \frac{(n-k+1)(k-2)}
         {k(n-k+3)}
    \sigma_{k-1}.
\]
Hence
\[
\begin{aligned}
    \sum_jF_j
    &=
    \frac{
        (n-k+1)\sigma_{k-1}
        -
        (n-k+3)s^2\sigma_{k-3}
    }{
        2s\sigma_{k-2}
    }
    \\
    &\geq
    \frac{n-k+1}{ks}
    \frac{\sigma_{k-1}}{\sigma_{k-2}}.
\end{aligned}
\]
Substituting this estimate into the preceding inequality proves the lemma.
\end{proof}

\subsection{The comparison Dirichlet problem}
\label{sec:comparison-existence}

For prescribed curvature quotient equations, the homogeneous
Dirichlet problem was studied by Lin and Trudinger \cite{LT},
and general boundary values were treated by Ivochkina, Lin, and
Trudinger \cite{ILT}. We use a finite normal-limit argument for
the boundary double-normal estimate; see \cite{IvoNormal}.

We prove the small-ball solvability result needed for the comparison
construction. The admissible radius depends only on $n,k$, and $K$.
Estimates used solely to construct the comparison solution may depend
on higher norms of the fixed boundary data.

\begin{proposition}\label{prop:comparison-existence}
Let $3\leq k<n$ and $K\geq1$, and set
\[
    R_*=
    \min\left\{
        1,
        \frac14
        \sqrt{
            \frac{\binom{n-1}{k}}
                 {\binom{n-1}{k-2}}
        },
        \frac{1}{2(1+K^2)}
        \sqrt{
            \frac{\binom{n-1}{k-1}}
                 {\binom{n-1}{k-3}}
        }
    \right\}>0.
\]
Suppose that $0<R\leq R_*$ and
$u\in C^\infty(\overline B_R)$ satisfies
\[
    \kappa[u]\in\Gamma_k,
    \qquad
    F(\kappa[u])=1
    \quad\text{in }B_R,
    \qquad
    |Du|\leq K.
\]
Then, for every $s\in(0,1)$ and $t\in[s,1]$, the Dirichlet problem
\begin{equation}\label{5.1}
    \begin{cases}
        F(\kappa[v_t])=t,\quad
        \kappa[v_t]\in\Gamma_k
            & \text{in }B_R,\\
        v_t=u
            & \text{on }\partial B_R
    \end{cases}
\end{equation}
has a unique solution
$v_t\in C^\infty(\overline B_R)$.
Moreover, $v_t\geq u$, and, for fixed $u,R,s$, the global derivative
bounds are uniform for $t\in[s,1]$. The constants in this existence
statement may depend on higher norms of $u$; they are not used in
the interior curvature estimate of Theorem~\ref{thm:main}.
\end{proposition}

\begin{proof}
Fix $s\in(0,1)$. We establish a priori estimates for smooth
admissible solutions of \eqref{5.1}, uniformly in $t\in[s,1]$.
Write
\[
    v=v_t,
    \qquad
    \phi=u|_{\partial B_R},
\]
and set
\[
\begin{gathered}
    p=Dv,
    \qquad
    Q=D^2v,
    \qquad
    W=\sqrt{1+|p|^2},
    \qquad
    \gamma=(I_n+p\otimes p)^{-\frac12},
    \\
    A=\frac1W\gamma Q\gamma,
    \qquad
    G(p,Q)
    =
    F\!\left(
        \frac1W\gamma Q\gamma
    \right),
    \qquad
    M=(M^{ij}),
    \quad
    M^{ij}
    =
    \frac{\partial G}{\partial Q_{ij}}.
\end{gathered}
\]
The matrix $A$ is a symmetric representation of the shape operator.
In this proof, constants may depend on $n,k,R,s$ and fixed smooth
norms of $u$, but not on $t$.

\medskip
\noindent\textit{Step 1. Height and gradient estimates.}

Choose $C_b>0$, depending on
$\|D^2u\|_{L^\infty}$, such that
\[
    \overline v
    =
    u+C_b(R^2-|x|^2),
    \qquad
    D^2\overline v<0.
\]
Comparison gives $u\leq v$. If
$v-\overline v$ had a positive interior maximum, then
\[
    D^2v
    \leq
    D^2\overline v<0
\]
there, contradicting admissibility. Thus
\begin{equation}\label{5.2}
    u\leq v\leq\overline v
    \qquad
    \text{in }B_R.
\end{equation}

Let $\mu$ denote the outward unit normal to $\partial B_R$.
Since all three functions agree on the boundary,
\[
    \overline v_\mu
    \leq
    v_\mu
    \leq
    u_\mu,
    \qquad
    D_\partial v
    =
    D_\partial\phi
    \quad\text{on }\partial B_R.
\]
This bounds the boundary gradient.

The vertical angle
\[
    q=W^{-1}
\]
satisfies
\[
    F^{ij}\nabla^2_{ij}q
    =
    -q\sum_iF_i\kappa_i^2
    \leq0.
\]
The minimum principle therefore bounds $|Dv|$ by its boundary maximum.
Denote the resulting global bound by $K_0$, and put
\[
    W_0=\sqrt{1+K_0^2}.
\]
Unlike the original $K$, the constant $K_0$ is allowed to depend
on higher norms of $u$.

\medskip
\noindent\textit{Step 2. Tangential control of the Hessian energy.}

In a principal orthonormal frame, write
\[
    f_i=F_i,
    \qquad
    E=\sum_i f_i\kappa_i^2,
    \qquad
    S=\sum_i f_i.
\]
Lemma~\ref{energy} gives
\begin{equation}\label{5.3}
    E
    \leq
    \frac{n+2}{k}
    \sum_{j\neq i}f_j\kappa_j^2
    +
    \frac{n+2-k}{n-k+1}t^2S,
    \qquad
    1\leq i\leq n.
\end{equation}

Let
\[
    F'(A)=(F^{ij}(A)).
\]
Since $F'(A)$ commutes with $A$,
\[
    M
    =
    \frac1W\gamma F'(A)\gamma,
    \qquad
    QMQ
    =
    W\gamma^{-1}F'(A)A^2\gamma^{-1},
\]
and hence
\[
    WE
    \leq
    \operatorname{tr}(MQ^2)
    \leq
    W^3E.
\]

For a unit vector $\omega\in\mathbb R^n$, set
\[
    P=I_n-\omega\otimes\omega,
    \qquad
    e=\frac{\gamma\omega}{|\gamma\omega|}.
\]
The matrix $\gamma^{-1}P\gamma^{-1}$ has kernel spanned by $e$,
and its nonzero eigenvalues are at least one. These coincide with
the eigenvalues of $P\gamma^{-2}P$ on $\omega^\perp$, where
\[
    P\gamma^{-2}P\geq P.
\]
Therefore
\[
\begin{aligned}
    \operatorname{tr}(MQPQ)
    &=
    W\operatorname{tr}
    \bigl(
        F'(A)A^2\gamma^{-1}P\gamma^{-1}
    \bigr)
    \\
    &\geq
    W\sum_i f_i\kappa_i^2(1-e_i^2)
    \\
    &\geq
    W\left(
        E-\max_i f_i\kappa_i^2
    \right).
\end{aligned}
\]
Applying \eqref{5.3} to an index realizing this maximum, and using
$t\leq1$ and $W\leq W_0$, we obtain
\begin{equation}\label{5.4}
    \operatorname{tr}(MQ^2)
    \leq
    C\operatorname{tr}(MQPQ)
    +
    C\operatorname{tr}M.
\end{equation}
Here we also used
\[
    S\leq W^3\operatorname{tr}M.
\]
Concavity and homogeneity give
\begin{equation}\label{5.5}
    \operatorname{tr}M
    \geq
    W^{-3}\operatorname{tr}F'(A)
    \geq
    W^{-3}F(I_n)
    \geq
    c(n,k,K_0)>0.
\end{equation}
Neither estimate uses a boundary second-derivative bound.

\medskip
\noindent\textit{Step 3. An auxiliary differential inequality.}

Differentiating the metric factors in $G$ gives
\[
    G_p
    =
    -\frac{t}{W^2}p
    -
    \frac{2}{W^2}MQp.
\]
Introduce the auxiliary operator
\begin{equation}\label{5.6}
    L
    =
    M^{ij}D_{ij}
    -
    \frac{t}{W^2}p\cdot D.
\end{equation}
Differentiation of
\[
    G(Dv,D^2v)=t
\]
then yields
\begin{equation}\label{5.7}
    Lp
    =
    \frac{2}{W^2}QMQp.
\end{equation}
Thus $L$ differs from the full linearization by the explicitly
separated term involving $MQp$.

Let $\xi,p_0$ be prescribed smooth vector fields, let $\lambda$ be a
prescribed smooth scalar function, and let $P$ be either $I_n$ or a
prescribed smooth orthogonal projection of rank $n-1$. For a fixed
$\varepsilon>0$, set
\[
    w
    =
    \xi\cdot p-\lambda
    -
    \varepsilon
    (p-p_0)^{\mathsf T}P(p-p_0).
\]
Assume that these fields and their derivatives through order two
are bounded. Writing
\[
    z
    =
    \xi-2\varepsilon P(p-p_0),
    \qquad
    Dw=Qz+c,
\]
the vector $c$ is bounded in terms of the prescribed fields and
$K_0$. Applying \eqref{5.7} and the product rule gives
\[
\begin{aligned}
    Lw
    &=
    \frac{2}{W^2}(Dw-c)^{\mathsf T}MQp
    -
    2\varepsilon\operatorname{tr}(MQPQ)
    +
    \mathcal R,
    \\
    |\mathcal R|
    &\leq
    C
    \left(
        \operatorname{tr}M
        +
        \sqrt{
            \operatorname{tr}M
            \operatorname{tr}(MQ^2)
        }
    \right).
\end{aligned}
\]
All terms involving second derivatives of $v$ in $\mathcal R$ are
linear in $Q$. Weighted Cauchy--Schwarz and Young's inequality
therefore give, for every $\varepsilon_0>0$,
\[
    Lw
    \leq
    -2\varepsilon\operatorname{tr}(MQPQ)
    +
    \varepsilon_0\operatorname{tr}(MQ^2)
    +
    C_{\varepsilon_0}
    \bigl(
        \operatorname{tr}M
        +
        M^{ij}w_iw_j
    \bigr).
\]
For $P=I_n$, choose $\varepsilon_0<\varepsilon$. For a rank-$n-1$
projection, choose $\varepsilon_0$ sufficiently small and apply
\eqref{5.4}. In either case,
\begin{equation}\label{5.8}
    Lw
    \leq
    C
    \bigl(
        \operatorname{tr}M
        +
        M^{ij}w_iw_j
    \bigr).
\end{equation}
Since $w$ is uniformly bounded, choosing a sufficiently large fixed
$b>0$ gives
\begin{equation}\label{5.9}
\begin{aligned}
    \widetilde w
    &=
    1-e^{-bw},
    \\
    L\widetilde w
    &=
    be^{-bw}
    \bigl(
        Lw-bM^{ij}w_iw_j
    \bigr)
    \leq
    C\operatorname{tr}M.
\end{aligned}
\end{equation}
In all applications below, the fields are constructed from fixed
boundary data and controlled scalar parameters, not from an
extension of the unknown full boundary gradient.

\medskip
\noindent\textit{Step 4. A boundary barrier.}

In a collar of $\partial B_R$, let
\[
    d=R-|x|,
    \qquad
    \mu=\frac{x}{|x|},
    \qquad
    P=I_n-\mu\otimes\mu.
\]
Here $d$ denotes the distance to the boundary of the base ball and
$\mu$ is the outward unit normal to $\partial B_R$.

Because $k<n$, the matrix
\[
    W^{-1}\gamma P\gamma
\]
belongs to $\Gamma_k$. Concavity and homogeneity give
\[
    M^{ij}P_{ij}
    \geq
    F(W^{-1}\gamma P\gamma).
\]
Its nonzero eigenvalues are $W^{-1}$ with multiplicity $n-2$ and
$\zeta W^{-1}$, where
\[
    \zeta
    =
    \frac{1+p_\mu^2}{W^2}
    \in(0,1].
\]
Set
\[
    c_b
    =
    \sqrt{
        \frac{\binom{n-1}{k}}
             {\binom{n-1}{k-2}}
    }>0.
\]
The elementary symmetric polynomials give
\[
    F(W^{-1}\gamma P\gamma)^2
    =
    \frac1{W^2}
    \frac{
        \binom{n-2}{k}
        +
        \zeta\binom{n-2}{k-1}
    }{
        \binom{n-2}{k-2}
        +
        \zeta\binom{n-2}{k-3}
    }
    \geq
    \frac{c_b^2\zeta}{W^2}.
\]
Indeed, after multiplication by the denominator, the difference
between the numerator and $c_b^2\zeta$ times the denominator is a
concave quadratic polynomial in $\zeta$, nonnegative at $0$ and
zero at $1$. Hence
\begin{equation}\label{5.10}
    M^{ij}P_{ij}
    \geq
    c_b
    \frac{\sqrt{1+p_\mu^2}}{W^2}.
\end{equation}

Since $R\leq\frac{c_b}{4}$ and $t\leq1$,
\[
    Ld
    =
    -\frac{M^{ij}P_{ij}}{|x|}
    +
    \frac{tp_\mu}{W^2}
    \leq
    -\frac{M^{ij}P_{ij}}{2|x|}.
\]
Fix $N\geq1$ and set
\[
    V
    =
    d-\frac N2d^2
\]
in a collar where
\[
    Nd\leq\frac12,
    \qquad
    d<\frac R2.
\]
Then
\begin{equation}\label{5.11}
\begin{aligned}
    LV
    &=
    (1-Nd)Ld
    -
    NM^{ij}\mu_i\mu_j
    \\
    &\leq
    -\frac1{4R}M^{ij}P_{ij}
    -
    NM^{ij}\mu_i\mu_j
    \\
    &\leq
    -c\operatorname{tr}M.
\end{aligned}
\end{equation}
Moreover, $V\geq0$, $V=0$ on $\partial B_R$, and $V>0$ on the
inner boundary of the collar.

\medskip
\noindent\textit{Step 5. Tangential and mixed second derivatives.}

Extend $\phi$ constantly along radial normals. For tangential
vectors $\tau,\tau'$ at a boundary point,
\[
    D^2v(\tau,\tau')
    =
    D_\partial^2\phi(\tau,\tau')
    +
    \frac{v_\mu}{R}
    \langle\tau,\tau'\rangle,
\]
so the pure tangential derivatives are bounded.

For the mixed derivatives, let $\xi$ be a fixed smooth vector field
tangent to the concentric spheres near
$y_0\in\partial B_R$, and set
\[
    w
    =
    \pm\xi\cdot D(v-\phi)
    -
    \varepsilon
    |PD(v-\phi)|^2.
\]
This has the form considered in Step 3. On the true boundary,
$w=0$ and the gradient of the square term vanishes.

In a fixed collar neighborhood of $y_0$, consider
\[
    \mathcal Q
    =
    \widetilde w
    +
    B|x-y_0|^2
    +
    C_0V.
\]
Choose $B$ so that $\mathcal Q\geq0$ on the artificial lateral
boundary, using the gradient bound. Then choose $C_0$ sufficiently
large so that the same inequality holds on the inner boundary and
$L\mathcal Q\leq0$ by \eqref{5.9} and \eqref{5.11}. The bounded
drift term in $L|x-y_0|^2$ is absorbed using \eqref{5.5}.
The minimum principle gives
\[
    \mathcal Q\geq0
\]
in the neighborhood.

Since $\mathcal Q(y_0)=0$ and
\[
    V_\mu(y_0)=-1,
\]
we obtain
\[
    0
    \geq
    (\mathcal Q)_\mu(y_0)
    =
    b\,w_\mu(y_0)-C_0.
\]
The two choices of sign give
\[
    |v_{\mu\xi}(y_0)|
    \leq
    C.
\]
All terms involving the prescribed fields are already bounded.

\medskip
\noindent\textit{Step 6. The finite normal limit.}

At $y\in\partial B_R$, write
\[
    p_T=D_\partial\phi
\]
and define
\[
\begin{gathered}
    C_y
    =
    (I_{n-1}+p_T\otimes p_T)^{-1},
    \\
    B_y(a)
    =
    C_y^{\frac12}
    \left(
        D_\partial^2\phi
        +
        \frac aR I_{n-1}
    \right)
    C_y^{\frac12},
    \qquad
    W_y(a)
    =
    \sqrt{1+|p_T|^2+a^2}.
\end{gathered}
\]
The scalar $a$ represents a candidate outward normal derivative.
Let
\[
    \mathcal F(B)
    =
    \left(
        \frac{\sigma_{k-1}(B)}
             {\sigma_{k-3}(B)}
    \right)^{\frac12},
    \qquad
    h_y(a)
    =
    \mathcal F(B_y(a))-tW_y(a).
\]
The interval on which
\[
    B_y(a)\in\Gamma_{k-1}
\]
has the form
\[
    (a_*(y),\infty).
\]
On this interval, $h_y$ is concave and
\[
    \lim_{a\to\infty}h_y'(a)
    =
    \mathcal F\left(\frac{C_y}{R}\right)-t.
\]
Since $|p_T|\leq K$,
\[
    \mathcal F\left(\frac{C_y}{R}\right)
    \geq
    \frac{1}{R(1+K^2)}
    \sqrt{
        \frac{\binom{n-1}{k-1}}
             {\binom{n-1}{k-3}}
    }
    \geq2.
\]
It follows that
\[
    h_y'(a)\geq1
\]
for every admissible $a$ and $t\in[s,1]$.

At the lower endpoint,
\[
    \lim_{a\downarrow a_*(y)}
    \mathcal F(B_y(a))
    =
    0,
\]
and therefore
\[
    \lim_{a\downarrow a_*(y)}
    h_y(a)
    =
    -tW_y(a_*(y))
    <
    0.
\]
On the other hand,
\[
    h_y(a)\to+\infty
    \qquad
    \text{as }a\to\infty.
\]
Hence there is a unique zero
\[
    \bar a(y,t)\in(a_*(y),\infty).
\]

Since
\[
    h_y'(\bar a)\geq1,
\]
the implicit function theorem gives smooth dependence of
$\bar a$ on $(y,t)$. This dependence is uniform on
$\partial B_R\times[s,1]$: the lower endpoint $a_*(y)$ is controlled
by the fixed tangential Hessian of $\phi$, the lower bound
$h_y'\geq1$ controls the zero, and at the zero
\[
    \mathcal F(B_y(\bar a))
    =
    tW_y(\bar a)
    \geq
    s.
\]
Newton--Maclaurin therefore keeps the bounded matrices
$B_y(\bar a)$ in a fixed compact subset of $\Gamma_{k-1}$.
All derivatives of $\bar a$ needed below are uniformly bounded.
We suppress $t$ from the notation.

Set
\[
    a_v=v_\mu(y).
\]
The tangential compression of the shape operator along the boundary
graph is represented by
\[
    \frac{B_y(a_v)}{W_y(a_v)}
\]
and belongs to $\Gamma_{k-1}$. Increasing $v_{\mu\mu}$ with all
other boundary data fixed produces a positive rank-one perturbation
of $A[v]$. Expansion of the elementary symmetric polynomials in that
entry gives
\begin{equation}\label{5.12}
    F_y^\infty(a_v)
    :=
    \lim_{\lambda\to\infty}
    G\bigl(
        Dv,
        D^2v+\lambda\mu\otimes\mu
    \bigr)
    =
    \frac{\mathcal F(B_y(a_v))}
         {W_y(a_v)}.
\end{equation}
Strict ellipticity gives
\[
    F_y^\infty(a_v)>t.
\]
Thus
\[
    h_y(a_v)>0,
    \qquad
    a_v>\bar a(y).
\]

\medskip
\noindent\textit{Step 7. A uniform normal-limit gap.}

We claim that there exists
$c_{\mathrm{rec}}>0$, independent of $t\in[s,1]$, such that
\begin{equation}\label{5.13}
    F_y^\infty(v_\mu(y))-t
    \geq
    c_{\mathrm{rec}},
    \qquad
    y\in\partial B_R.
\end{equation}

Set
\[
    \varepsilon_*
    :=
    \min_{\partial B_R}
    \bigl(
        v_\mu-\bar a
    \bigr)
    >0,
\]
and choose a minimizing point $y_0$. Extend $\bar a$ constantly
along radial normals and define
\[
    p_0
    =
    D_\partial\phi
    +
    (\bar a+\varepsilon_*)\mu,
\]
and
\[
    w
    =
    v_\mu
    -
    \bar a
    -
    \varepsilon_*
    -
    \varepsilon
    |Dv-p_0|^2.
\]
The parameter $\varepsilon_*$ is bounded, so the derivatives of
$p_0$ are uniformly controlled. On the true boundary, if
\[
    a
    =
    v_\mu
    -
    \bar a
    -
    \varepsilon_*
    \geq0,
\]
then
\[
    w
    =
    a-\varepsilon a^2
    \geq0
\]
for a sufficiently small fixed $\varepsilon>0$.
At $y_0$,
\[
    w=0,
    \qquad
    Dv=p_0.
\]

Apply Step 3 with $P=I_n$ and use the barrier from Step 5.
It follows that
\[
    w_\mu(y_0)\leq C.
\]
Since the prescribed normal extensions satisfy
\[
    D_\mu\mu=0,
    \qquad
    D_\mu\bar a=0,
\]
and the quadratic term has zero derivative at $y_0$, we obtain
\[
    v_{\mu\mu}(y_0)
    \leq
    C.
\]
The tangential and mixed estimates, together with positivity of the
mean curvature, give a lower bound as well. Hence the full boundary
Hessian is bounded at $y_0$.

Bounded admissible curvature matrices with $F\geq s$ remain in a
compact subset of $\Gamma_k$. Indeed, the Newton--Maclaurin inequalities give a positive lower bound
for $\sigma_k$, and hence prevent the curvature vector from
approaching $\partial\Gamma_k$ under the uniform curvature bound. Strict ellipticity is therefore
uniform on this compact set. Consequently, by uniform ellipticity along the positive rank-one
segment,
\[
\begin{aligned}
& G\bigl(p,Q+\mu\otimes\mu\bigr)-G(p,Q)
\\
&\qquad
=
\int_0^1
M^{ij}\bigl(p,Q+\tau\mu\otimes\mu\bigr)
\mu_i\mu_j\,d\tau
\geq c>0,
\end{aligned}
\]
where
\(
M^{ij}=\frac{\partial G}{\partial Q_{ij}}.
\)
 Passing to the normal limit gives
\[
    F_{y_0}^\infty
    \bigl(
        v_\mu(y_0)
    \bigr)
    -
    t
    \geq
    c_0>0.
\]
On \([\bar a(y_0),v_\mu(y_0)]\), concavity gives
\[
    h_{y_0}'
    \leq
    h_{y_0}'(\bar a(y_0))
    \leq
    C.
\]
Consequently,
\[
\begin{aligned}
    c_0
    &\leq
    W_{y_0}
    \bigl(
        v_\mu(y_0)
    \bigr)
    \left[
        F_{y_0}^\infty
        \bigl(
            v_\mu(y_0)
        \bigr)
        -
        t
    \right]
    \\
    &=
    h_{y_0}
    \bigl(
        v_\mu(y_0)
    \bigr)
    \\
    &\leq
    C
    \bigl(
        v_\mu(y_0)
        -
        \bar a(y_0)
    \bigr)=
    C\varepsilon_*.
\end{aligned}
\]
Hence
\[
    \varepsilon_*\geq c_1>0.
\]
For every boundary point, $h_y'\geq1$ and
\[
    W_y(v_\mu)\leq W_0
\]
give
\[
    F_y^\infty(v_\mu)-t
    =
    \frac{h_y(v_\mu)}
         {W_y(v_\mu)}
    \geq
    \frac{v_\mu-\bar a}{W_0}
    \geq
    \frac{c_1}{W_0}.
\]
Thus \eqref{5.13} holds with
\[
    c_{\mathrm{rec}}
    =
    \frac{c_1}{W_0}.
\]

We now use this gap to bound $v_{\mu\mu}$ on the whole boundary.
Write
\[
    Q
    =
    Q_0
    +
    v_{\mu\mu}\mu\otimes\mu,
\]
and
\[
    A
    =
    A_0
    +
    v_{\mu\mu}\zeta\otimes\zeta,
    \qquad
    A_0
    =
    W^{-1}\gamma Q_0\gamma,
    \qquad
    \zeta
    =
    W^{-\frac12}\gamma\mu.
\]
Here $Q_0$ and $A_0$ are bounded by the tangential and mixed
estimates. The compression $A_T$ of $A_0$ to $\zeta^\perp$ has
the same eigenvalues as
\[
    \frac{B_y(a_v)}{W_y(a_v)}.
\]
These matrices range in a compact subset of $\Gamma_{k-1}$.

The rank-one expansion of the equation gives
\[
    v_{\mu\mu}|\zeta|^2
    \left(
        \sigma_{k-1}(A_T)
        -
        t^2\sigma_{k-3}(A_T)
    \right)
    =
    t^2\sigma_{k-2}(A_0)
    -
    \sigma_k(A_0).
\]
The right-hand side is bounded, whereas the coefficient of
$v_{\mu\mu}$ is bounded below by a positive constant. Indeed,
\[
    |\zeta|^2
    \geq
    W_0^{-3},
\]
and
\[
\begin{aligned}
    \sigma_{k-1}(A_T)
    -
    t^2\sigma_{k-3}(A_T)
    &=
    \sigma_{k-3}(A_T)
    \left[
        \bigl(F_y^\infty(a_v)\bigr)^2
        -
        t^2
    \right]
    \\
    &\geq
    c>0.
\end{aligned}
\]
For $k=3$, the factor $\sigma_{k-3}$ is $\sigma_0=1$.
Thus
\[
    |v_{\mu\mu}|
    \leq
    C
    \qquad
    \text{on }\partial B_R,
\]
completing the boundary second-derivative estimate.

\medskip
\noindent\textit{Step 8. Global estimates and the continuity method.}

Let
\[
    q=W^{-1},
    \qquad
    \underline q=\frac{1}{2W_0},
\]
and consider
\[
    \log\kappa_1-\log(q-\underline q).
\]
If its maximum were attained in the interior, choose a unit principal
vector realizing $\kappa_1$ and extend it locally with vanishing first
covariant derivative and vanishing symmetrized second covariant
derivative at the contact point. The component $h_{11}$ touches
$\kappa_1$ from below, so the corresponding smooth quotient has a
local maximum there.

With
\[
    b=\log h_{11},
\]
the Simons identity and concavity give
\[
    \Delta_Fb
    \geq
    t\kappa_1
    -
    E
    -
    \sum_i f_i b_i^2.
\]
The angle identity gives
\[
    \Delta_F[-\log(q-\underline q)]
    =
    \frac{q}{q-\underline q}E
    +
    \sum_i f_i
    \frac{q_i^2}
         {(q-\underline q)^2}.
\]
At the maximum,
\[
    b_i
    =
    \frac{q_i}{q-\underline q},
\]
so the gradient terms cancel. Therefore
\[
    0
    \geq
    t\kappa_1
    +
    \frac{\underline q}{q-\underline q}E
    >
    0,
\]
a contradiction.

The maximum is therefore attained on the boundary. This argument
also covers repeated largest curvatures, since it differentiates a
touching component rather than the largest eigenvalue itself.
The boundary estimate bounds $\kappa_1$ globally.

Since
\[
    \Gamma_k\subset\Gamma_2,
\]
we have
\[
    |\kappa|^2
    =
    \sigma_1(\kappa)^2
    -
    2\sigma_2(\kappa)
    \leq
    \sigma_1(\kappa)^2
    \leq
    n^2\kappa_1^2.
\]
The gradient bound and
\[
    |D^2v|
    \leq
    W^3|\kappa[v]|
\]
then give a global $C^2$ estimate.

Use the path
\[
    t
    =
    1-\tau(1-s),
    \qquad
    0\leq\tau\leq1,
\]
starting from $v=u$. All the preceding estimates are uniform in
$\tau$. A bounded family in $\Gamma_k$ satisfying $F\geq s>0$
stays a positive distance away from $\partial\Gamma_k$, so the equations are uniformly
elliptic. The operator is smooth and concave in $Q$.

Global $C^{2,\alpha}$ estimates for uniformly elliptic concave
Dirichlet equations, followed by Schauder estimates, give the higher
bounds needed for closedness; see \cite{GT}. The full linearized
Dirichlet operator
\[
    M^{ij}D_{ij}
    +
    G_{p_i}D_i
\]
has no zeroth-order term and is invertible by the maximum principle
and linear Schauder theory. Thus the solution set is open and closed
along the path. This proves existence for every $t\in[s,1]$.
Uniqueness follows from the admissible comparison principle.
\end{proof}

\begin{remark}
The construction does not require a positive lower bound for the
product of the eigenvalues of the linearization on a fixed level set.
In particular, condition (1.10) of \cite{LiraCruz} cannot be used for
the present quotient without further restrictions.

For $3\leq k<n$, set
\[
    \lambda(L)
    =
    (L,1,\ldots,1).
\]
Then
\[
    F(\lambda(L))^2
    =
    \frac{
        \binom{n-1}{k}
        +
        L\binom{n-1}{k-1}
    }{
        \binom{n-1}{k-2}
        +
        L\binom{n-1}{k-3}
    }
    \longrightarrow
    \frac{
        \binom{n-1}{k-1}
    }{
        \binom{n-1}{k-3}
    }
    >0.
\]
Differentiation shows that
\[
    F_1(\lambda(L))
\]
is of order $L^{-2}$, whereas
\[
    F_i(\lambda(L))
    =
    O(1),
    \qquad
    i>1.
\]
Rescaling $\lambda(L)$ to any fixed positive $F$-level leaves these
derivatives unchanged, because $DF$ is homogeneous of degree zero.
Hence
\[
    \prod_{i=1}^nF_i(\lambda(L))
    \longrightarrow
    0.
\]
\end{remark}

\subsection{Proof of the main theorem}

\begin{proof}[Proof of Theorem~\ref{thm:main}]
We first derive the local curvature estimate on a ball for which the
admissible comparison solution $v$ in \eqref{3.1} has been constructed.
Write $R_0=\frac{3R}{4}$. The comparison principle gives
\[
    v>u
    \qquad\text{in }B_{R_0}.
\]
Moreover, since
\[
    \kappa[v]\in\Gamma_k\subset\Gamma_1,
\]
the mean curvature of $\Sigma_v$ is positive. Hence the maximum
principle gives
\[
    v
    \leq
    \max_{\partial B_{R_0}}v
    =
    \max_{\partial B_{R_0}}u.
\]
Therefore
\[
    u<v\leq\max_{\partial B_{R_0}}u
    \qquad\text{in }B_{R_0}.
\]
In particular,
\[
    \|v\|_{L^\infty(B_{R_0})}\leq K,
    \qquad
    \operatorname{osc}_{B_{R_0}}v\leq2K.
\]

Lemma~\ref{gradient} supplies the interior gradient bounds required
in Theorem~\ref{thm:separation}. By \eqref{3.24}, after decreasing
the separation constant if necessary, there exists
\[
    c_0=c_0(n,k,R,K)\in(0,1]
\]
such that
\[
    \rho:=v-u\geq c_0
    \qquad\text{in }B_{\frac R2}.
\]

We next obtain a boundary decay estimate for $\rho$, using the
concave barrier from \cite{QiuYan}. For $y\in\partial B_{R_0}$, set
\[
    w_y(x)
    =
    u(y)+K\sqrt{2(R_0^2-x\cdot y)},
    \qquad
    x\in\overline B_{R_0}.
\]
For $z\in\partial B_{R_0}$,
\[
    w_y(z)
    =
    u(y)+K|z-y|
    \geq
    u(z)
    =
    v(z),
\]
where we used the boundary condition $v=u$ on $\partial B_{R_0}$.
Moreover, $w_y$ is smooth in $B_{R_0}$ and
\[
    D^2w_y(x)
    =
    -\frac{K}
    {\bigl[2(R_0^2-x\cdot y)\bigr]^{\frac32}}
    y\otimes y
    \leq0.
\]
If $v-w_y$ had a positive maximum, it would occur at an interior
point, where
\[
    D^2v\leq D^2w_y\leq0.
\]
Hence the second fundamental form of $\Sigma_v$ would be
nonpositive there, and in particular
\[
    H[v]\leq0,
\]
contradicting $H[v]>0$. Thus
\[
    v\leq w_y
    \qquad\text{in }B_{R_0}.
\]

For $x\in B_{R_0}$, write
\[
    d_\partial(x)
    =
    \operatorname{dist}(x,\partial B_{R_0})
    =
    R_0-|x|,
\]
and choose a nearest boundary point $y$. Then
\[
    |x-y|=d_\partial(x),
\]
and
\[
    R_0^2-x\cdot y
    =
    R_0d_\partial(x).
\]
Consequently,
\[
\begin{aligned}
    0\leq\rho(x)
    &\leq
    w_y(x)-u(x)
    \\
    &\leq
    K\sqrt{2R_0d_\partial(x)}
    +
    Kd_\partial(x)\leq C_0K\,d_\partial(x)^{\frac12},
\end{aligned}
\]
where $C_0=(\sqrt2+1)\sqrt{R_0}$.

Fix $\delta_1\in\left(\frac{c_0}{4},\frac{c_0}{2}\right)$, and let $\Omega_{\delta_1}$ be the connected component of $\{x\in B_{R_0}:\rho(x)>\delta_1\}$ containing $B_{\frac R2}$. The preceding boundary estimate implies
\[
    \operatorname{dist}
    \bigl(
        \overline{\Omega}_{\delta_1},
        \partial B_{R_0}
    \bigr)
    \geq
    d_0,
\]
where $ d_0=\left( \frac{\delta_1}{C_0K}\right)^2>0$. Since $\delta_1>\frac{c_0}{4}$, the constant $d_0$ has a positive lower bound depending only on $n,k,R,K$. Thus
\[
    \Omega_{\delta_1}\Subset B_{R_0},
\]
and
\[
    \rho=\delta_1
    \qquad
    \text{on }\partial\Omega_{\delta_1}.
\]

Applying Lemma~\ref{gradient} on the balls
\[
    B_{\frac{d_0}{2}}(x),
    \qquad
    x\in\overline{\Omega}_{\delta_1},
\]
gives
\[
    \|Dv\|_{L^\infty(\Omega_{\delta_1})}
    \leq
    C(n,k,R,K).
\]
The bound is uniform because $d_0$ is bounded below in terms of the
same data.

Set
\[
    v_{\delta_1}=v-\delta_1.
\]
Vertical translation preserves the principal curvatures, so
$v_{\delta_1}$ is admissible and
\[
\begin{gathered}
    v_{\delta_1}-u
    =
    \rho-\delta_1
    >
    0
    \quad\text{in }\Omega_{\delta_1},
    \qquad
    v_{\delta_1}=u
    \quad\text{on }\partial\Omega_{\delta_1},
    \\
    \|u\|_{C^1(\Omega_{\delta_1})}
    +
    \|v_{\delta_1}\|_{C^1(\Omega_{\delta_1})}
    \leq
    K_1(n,k,R,K).
\end{gathered}
\]
No boundary regularity of $\Omega_{\delta_1}$ is required here,
since the proof of Proposition~\ref{pog} uses only an interior
maximum argument and $u,v_{\delta_1}$ are smooth in a neighborhood
of $\overline{\Omega}_{\delta_1}$.

Applying Proposition~\ref{pog}, we obtain
\[
    \sup_{\Omega_{\delta_1}}
    (\rho-\delta_1)^\beta
    \kappa_1[u]
    \leq
    C,
\]
where $\beta$ and $C$ depend only on $n,k,R,K$. Since
\[
    \rho-\delta_1
    \geq
    \frac{c_0}{2}
    \qquad
    \text{on }B_{\frac R2},
\]
it follows that
\[
    \sup_{B_{\frac R2}}
    \kappa_1[u]
    \leq
    C
    \left(
        \frac{2}{c_0}
    \right)^\beta
    \leq
    C(n,k,R,K).
\]

To bound the remaining principal curvatures, use
\[
    H=\sum_{j=1}^n\kappa_j>0
\]
and $\kappa_j\leq\kappa_1$. For each $1\leq i\leq n$,
\[
\begin{aligned}
    \kappa_i
    &=
    H-\sum_{j\neq i}\kappa_j
    \\
    &>
    -\sum_{j\neq i}\kappa_j\geq -(n-1)\kappa_1.
\end{aligned}
\]
Hence all principal curvatures are bounded on $B_{\frac R2}$.

In graph coordinates, let
\[
    h_u=(h^u_{ij}),
    \qquad
    g_u=I_n+Du\otimes Du.
\]
Then
\[
    D^2u=W_uh_u,
    \qquad
    \|g_u\|_{\mathrm{op}}\leq W_u^2,
\]
and the symmetric matrix
\[
    g_u^{-\frac12}h_ug_u^{-\frac12}
\]
has eigenvalues $\kappa[u]$. Therefore
\[
\begin{aligned}
    |D^2u|
    &=
    W_u|h_u|
    \\
    &\leq
    W_u\|g_u\|_{\mathrm{op}}
    \left|
        g_u^{-\frac12}h_ug_u^{-\frac12}
    \right|
    \\
    &\leq
    W_u^3|\kappa[u]|
    \leq
    (1+K^2)^{\frac32}
    |\kappa[u]|.
\end{aligned}
\]
Here $|\cdot|$ denotes the Euclidean norm for vectors and the
Frobenius norm for matrices. This proves \eqref{1.3} whenever the
comparison graph is available on $B_{R_0}$.

It remains to construct comparison graphs on balls of a uniform
local radius. If $u$ is assumed only to belong to $C^4(B_R)$, its
admissibility makes the linearized equation uniformly elliptic on
every compact subdomain for this fixed solution. Since the operator
is smooth, differentiation of the equation and interior Schauder
estimates imply
\[
    u\in C^\infty_{\mathrm{loc}}(B_R).
\]
This is only a qualitative regularity statement; no bound obtained
from this argument is used in the curvature estimate.

Fix $x_0\in B_{\frac R2}$ and set
\[
    r_0
    =
    \min\left\{
        \frac R8,
        \frac{R_*(n,k,K)}2
    \right\},
\]
where $R_*$ is given by
Proposition~\ref{prop:comparison-existence}. Since
\[
    \overline B_{\frac{4r_0}{3}}(x_0)
    \Subset
    B_R,
\]
the restriction of $u$ is smooth in a neighborhood of this closed
ball and supplies the boundary data required by that proposition.
We therefore obtain an admissible comparison solution on $B_{r_0}(x_0)$
with right-hand side $s_0$ and boundary values $u$.

Apply the preceding argument after translation, with the outer
radius $R$ replaced by $\frac{4r_0}{3}$. Its comparison radius is then $r_0$, and its target ball is $B_{\frac{2r_0}{3}}(x_0)$.

Consequently,
\[
    \sup_{B_{\frac{2r_0}{3}}(x_0)}
    \bigl(
        |\kappa[u]|+|D^2u|
    \bigr)
    \leq
    C(n,k,r_0,K).
\]

Higher norms of the boundary data may enter the construction of
each comparison solution. However, the separation and Pogorelov
estimates above use only $n,k,r_0,K$. Since $r_0$ depends only on
$n,k,R,K$ and is independent of $x_0$, evaluating this estimate at
each $x_0\in B_{\frac R2}$ proves \eqref{1.3}.
\end{proof}

\begin{remark}
The augmented vector
\[
    \left(
        \frac{\kappa[v]}{s_0},
        1
    \right)
\]
used in the doubling argument is an algebraic construction; no
hypersurface in $\mathbb R^{n+2}$ is introduced in the proof.
In particular, appending a component equal to $1$ should not be
identified with adjoining a quadratic graph variable.

Indeed, let
\[
    \widetilde u(x,t)
    =
    u(x)+\frac{t^2}{2}.
\]
At $t=0$, the induced metric and second fundamental form of this
graph, with the same downward-normal convention, are
\[
    \widetilde g(x,0)
    =
    \begin{pmatrix}
        g_u(x)&0\\
        0&1
    \end{pmatrix},
    \qquad
    \widetilde h(x,0)
    =
    \begin{pmatrix}
        h_u(x)&0\\
        0&W_u(x)^{-1}
    \end{pmatrix}.
\]
Thus, up to reordering,
\[
    \kappa[\widetilde u](x,0)
    =
    \left(
        \kappa[u](x),
        \frac{1}{W_u(x)}
    \right),
\]
so the additional principal curvature is $W_u(x)^{-1}$ rather than
$1$ in general.
\end{remark}

\section*{Declarations}
\textbf{Funding.} This work was supported by the ``University of Science and
Technology of China--Xinjiang Normal University Cooperative Development
Joint Fund'' (Project No.~XJNULH2501) and the National Natural Science
Foundation of China (Regional Science Foundation Project, Grant
No.~12661041).

\textbf{Use of artificial intelligence.} The authors used ChatGPT (OpenAI)
in the development and checking of some algebraic and differential-geometric
calculations and proof steps, and for language editing. All arguments,
calculations, and references in the final manuscript were independently
verified by the authors, who take full responsibility for the content.

\textbf{Conflict of interest.} The authors declare that they have no conflict of interest.

\begingroup
\small
\raggedright

\endgroup
\printauthorinformation
\end{document}